%% file: lsrk_cr.tex
\documentclass[preprint,letterpaper,10pt]{elsarticle}

\usepackage{amsmath}
\usepackage{amssymb}
\usepackage{amsthm}

\usepackage{empheq}
\usepackage{xcolor}

\usepackage[margin=1.2in]{geometry}

\usepackage[colorlinks=true,linkcolor=blue,filecolor=blue,urlcolor=blue,citecolor=blue,pdftex,plainpages=false]{hyperref}

\usepackage{verbatim}

\usepackage{lineno}

\newdefinition{defi}{Definition}
\newtheorem{lem}{Lemma}
\newtheorem{thm}{Theorem}
\newdefinition{rmk}{Remark}

\begin{document}

\begin{frontmatter}

\title{2N-storage Runge-Kutta methods:\\
$c$-reflection symmetry and factorization of the Butcher tableau}

\author{Alexei Bazavov}
\ead{bazavov@msu.edu}
\address{Department of Computational Mathematics, Science and Engineering and\\
Department of Physics and Astronomy,\\
Michigan State University, East Lansing, MI 48824, USA}

\begin{abstract}
Low-storage Runge-Kutta schemes of Williamson's type, so-called 2N-storage schemes, are further examined as a follow-up to the recent work. It is found that the augmented Butcher tableau factorizes into a product of matrices with special properties. Those properties reveal that the 2N-storage methods of the order of global accuracy less than five possess a symmetry, called $c$-reflection symmetry, \textit{i.e.} most methods exist in pairs. A transformation that relates the Butcher tableaux of the pairs is found and the fact that the $c$-reflected method satisfies the same order conditions as the original one is proven. Numerical evidence that validates the analytic results is presented. Branches of solutions for (5,4) methods, first explored by Carpenter and Kennedy, are constructed numerically. Four new (5,4) schemes with coefficients expressed in radicals and one with rational coefficients are examined for illustration. Eight new (6,4) schemes, some of which can be expressed in rationals or radicals, and one (8,4) scheme, are studied to understand the practical implications of the $c$-reflection symmetry for methods with higher number of stages.
In the absence of closed-form analytic solutions for 2N-storage Runge-Kutta methods of order four and above, the general symmetry properties, as well as some specific analytic solutions presented here, may help in development and optimization of 2N-storage schemes.
\end{abstract}


\end{frontmatter}


\section{Introduction}
\label{sec_intro}
\input{sec_intro.tex}

\section{2N-storage Runge-Kutta methods}
\label{sec_rk}
\input{sec_rk.tex}

\section{Augmented form of the Butcher tableau and order conditions}
\label{sec_augm_tab}
\input{sec_augm_tab.tex}

\section{Factorization of the Butcher tableau for 2N-storage methods}
\label{sec_factor}
\input{sec_factor.tex}

\section{$c$-reflection symmetry and  order conditions}
\label{sec_csym}
\input{sec_csym.tex}

\section{Numerical experiments}
\label{sec_num}
\input{sec_num.tex}

\section{Conclusions}
\label{sec_concl}
\input{sec_concl.tex}

\bigskip
\textbf{Acknowledgements.}
The author expresses deep gratitude to the CERN Theory Division for hospitality and financial support where a large fraction of this work was carried out during the sabbatical leave, and, in particular, to Matteo Di Carlo, Felix Erben, Philippe de Forcrand, Andreas J\"{u}ttner, Simon Kuberski and Tobias Tsang. The author is indebted to Andreas von Manteuffel  for illuminating discussions on the subject of Groebner bases and general symbolic and numerical tools for solving systems of nonlinear equations, and to Leon Hostetler for careful reading and comments on the manuscript.
This work was in part supported
by the U.S. National Science Foundation under award
PHY23-09946.

\appendix

\section{Order conditions for a Runge-Kutta method up to and including order four}
\label{sec_app_oc}
\input{sec_app_oc.tex}

\section{Order conditions for a (5,4) 2N-storage Runge-Kutta method in the $d$-form}
\label{sec_app_dform}
\input{sec_app_dform.tex}

\section{Independent checks of the (5,4) solutions}
\label{sec_app_checks54}
\input{sec_app_checks54.tex}

\section{Matlab script for constructing several $c$-reflected methods}
\label{sec_app_matlab}
\input{sec_app_matlab.tex}

\bibliography{lsrk_cr}

\end{document}

%% file: sec_intro.tex
This paper is a follow-up to the recent work~\cite{Bazavov2025a}, where some new relations for 2N-storage Runge-Kutta methods of Williamson~\cite{WILLIAMSON198048} were uncovered. Those methods are reformulations of standard Runge-Kutta methods where the information at intermediate stages can be overwritten and the memory footprint for 2N-storage methods is smaller than for the standard ones. The focus of this work, however, is not on the computational or efficiency aspects, but the structural properties of the 2N-storage methods. The new features presented here may help in the development of new or refinement of the existing 2N-storage Runge-Kutta methods. It may also be useful to understand if they are unique to 2N-storage methods of Williamson's type or if other types of Runge-Kutta methods possess them.

It is probably easiest to start with an observation. Carpenter and Kennedy in Ref.~\cite{CK1994} presented several (4,3) and (5,4) low-storage methods. Of particular interest are the pair of (4,3) methods shown in Table~\ref{tab_43_CK} and the four methods shown in Table 1 of Ref.~\cite{CK1994} which are also reproduced in Table~\ref{tab_54_CK}. The notation $(s,p)$ represents an $s$-stage method of the order of global accuracy $p$.

From Table~\ref{tab_43_CK} one can notice that for the pair of (4,3) methods the $c_i$ coefficients of one method are equal to $1-c_i$ in reverse order for the other method. The same holds for the pairs ``SOLUTION 1'' and ``SOLUTION 2'', and ``SOLUTION 3'' and ``SOLUTION 4'' of the (5,4) pairs shown in Table~\ref{tab_54_CK}. (This may not be immediately apparent due to the floating point format, however, this is the case up to the precision with which the coefficients were computed.) This poses several questions:
\begin{itemize}
\item[Q1:] Is this accidental or is there some symmetry?
\item[Q2:] Is there a transformation that relates the Butcher tableaux of these pairs?
\item [Q3:] Is there a reason these methods were found in pairs in the first place?
\end{itemize}
Note that similar relations for the nodes $c_i$ are observed for methods that are called adjoint~\cite{HairerBook1} or reflected~\cite{ButcherBook} in the literature. The methods discussed here are completely different from those.

\begin{table}[t]
\centering
\parbox{.35\linewidth}{
\[
\begin{array}{c|cccc}
0 & & & \\[1.5mm]
\frac{1}{9} & \phantom{-}\frac{1}{9} & & & \\[1.5mm]
\frac{4}{9} & -\frac{11}{36} & \phantom{-}\frac{3}{4} & & \\[1.5mm]
\frac{6}{9} & -\frac{1}{12} & \phantom{-}\frac{7}{20} & \phantom{-}\frac{2}{5} & \\[1.5mm]
\hline\\[-4mm]
& -1 & 2 & -\frac{5}{4} & \phantom{-}\frac{5}{4}
\end{array}
\]
}
\parbox{0.1\linewidth}{
\[
\begin{array}{c|c}
\phantom{-}0 & \frac{1}{9} \\[1.5mm]
-\frac{5}{9} & \frac{3}{4} \\[1.5mm]
-1 & \frac{2}{5} \\[1.5mm]
-\frac{33}{25} & \frac{5}{4} \\[1.5mm]
\end{array}
\]
}\hspace{0.05\linewidth}
\parbox{.35\linewidth}{
\[
\begin{array}{c|cccc}
0 & & & \\[1.5mm]
\frac{3}{9} & \phantom{-}\frac{3}{9} & & & \\[1.5mm]
\frac{5}{9} & -\frac{5}{18} & \phantom{-}\frac{5}{6} & & \\[1.5mm]
\frac{8}{9} & \phantom{-}\frac{41}{90} & -\frac{1}{6} & \phantom{-}\frac{3}{5} & \\[1.5mm]
\hline\\[-4mm]
& \phantom{-}\frac{3}{20} & \phantom{-}\frac{1}{4} & \phantom{-}\frac{7}{20} & \phantom{-}\frac{1}{4}
\end{array}
\]
}
\parbox{0.1\linewidth}{
\[
\begin{array}{c|c}
\phantom{-}0 & \frac{1}{3} \\[1.5mm]
-\frac{11}{15} & \frac{5}{6} \\[1.5mm]
-\frac{5}{3} & \frac{3}{5} \\[1.5mm]
-1 & \frac{1}{4} \\[1.5mm]
\end{array}
\]
}
\caption{A pair of (4,3) 2N-storage methods of Ref.~\cite{CK1994}.\label{tab_43_CK}}
\end{table}

\begin{table}[t]
\centering
\begin{tabular}{c|c|c|c|c}
COEF & SOLUTION 1 & SOLUTION 2 &
SOLUTION 3 & SOLUTION 4 \\
\hline
$A_1$  & $0$  & $0$ & $0$ & $0$ \\
$A_2$ & $-0.4812317431372$ & $-0.4801594388478$ & $-0.4178904745$ & $-0.7274361725534$ \\
$A_3$ & $-1.049562606709$ & $-1.4042471952$ & $-1.192151694643$ & $-1.906288083353$ \\
$A_4$ & $-1.602529574275$ & $-2.016477077503$ & $-1.697784692471$ & $-1.444507585809$ \\
$A_5$ & $-1.778267193916$ & $-1.056444269767$ & $-1.514183444257$ & $-1.365489400418$ \\
$B_1$ & $9.7618354692056E-2$ & $0.1028639988105$ & $0.1496590219993$ & $4.1717869324523E-2$ \\
$B_2$ & $0.4122532929155$ & $0.7408540575767$ & $0.3792103129999$ & $1.232835518522$ \\
$B_3$ & $0.4402169639311$ & $0.7426530946684$ & $0.8229550293869$ & $0.5242444514624$ \\
$B_4$ & $1.426311463224$ & $0.4694937902358$ & $0.6994504559488$ & $0.7212913223969$ \\
$B_5$ & $0.1978760537318$ & $0.1881733382888$ & $0.1530572479681$ & $0.2570977031703$ \\
$c_1$ & $0$ & $0$ & $0$ & $0$ \\
$c_2$ & $9.7618354692056E-2$ & $0.1028639988105$ & $0.1496590219993$ & $4.1717869324523E-2$  \\
$c_3$ & $0.3114822768438$ & $0.487989987833$ & $0.3704009573644$ & $0.377744236865$ \\
$c_4$ & $0.5120100121666$ & $0.6885177231562$ & $0.6222557631345$ & $0.6295990426348$ \\
$c_5$ & $0.8971360011895$ & $0.9023816453077$ & $0.9582821306748$ & $0.8503409780005$
\end{tabular}
\caption{Four (5,4) 2N-storage methods of Ref.~\cite{CK1994}.
\label{tab_54_CK}
}
\end{table}

A hint is present in the original paper by Williamson~\cite{WILLIAMSON198048} that introduced this particular 2N-storage format. For (3,3) methods an analytic solution is available and the 2N-storage constraint can be formulated as a single equation involving $c_2$ and $c_3$:
\begin{equation}
c_3^2(1-c_2)+c_3\left(c_2^2+\frac{1}{2}c_2-1\right)
+\left(\frac{1}{3}-\frac{1}{2}c_2\right)=0.\label{eq_Wc3}
\end{equation}
$c_2$ can be taken as a free parameter, Eq.~(\ref{eq_Wc3}) solved for $c_3$ and the rest of the Butcher tableau constructed from $c_2$ and $c_3$ with the well-known formulas for standard (3,3) Runge-Kutta methods. As was noted in Ref.~\cite{Bazavov2021}, and has probably been known to practitioners, Eq.~(\ref{eq_Wc3}) has a reflection symmetry with respect to $c_2+c_3=1$ line. There are three branches of solutions that are shown in Fig.~\ref{fig_Wcurves}, called the Williamson curves. Usually, discussion is limited to the $0\leqslant c_2,c_3\leqslant1$ region, and only one branch is shown, in the region enclosed in the box close to the center of the figure. For the discussion here, although one is unlikely to be interested in methods with $c_2<-2.7$ or $c_3>3.7$, it is important to take a note of that symmetry of (3,3) methods.

\begin{figure}[t]
\centering
\includegraphics[width=0.6\textwidth]{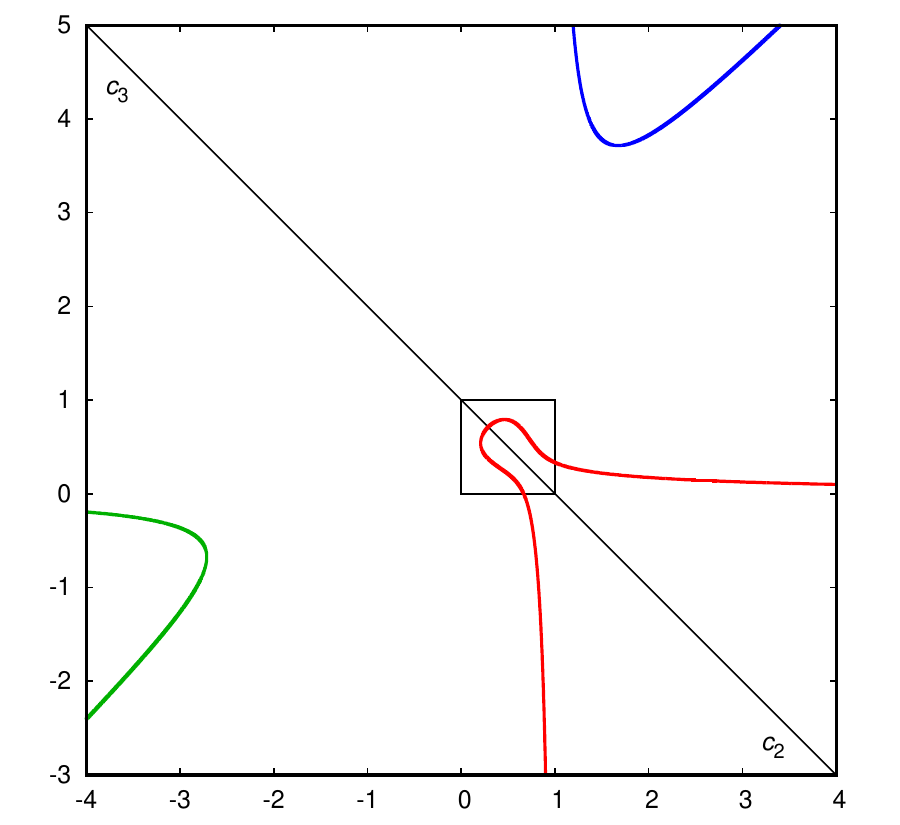}
\caption{The three branches of solutions of Eq.~(\ref{eq_Wc3}) that completely classify (3,3) 2N-storage methods in the $c_2-c_3$ plane.
\label{fig_Wcurves}
}
\end{figure}

The paper is structured as follows. In Sec.~\ref{sec_rk} well-known and recently derived~\cite{Bazavov2025a} results for 2N-storage Runge-Kutta methods are briefly reviewed. A comprehensive account of the properties of 2N-storage methods of Williamson's type has been recently presented in Ref.~\cite{Bazavov2025a} and it may be helpful, although not necessary, for the reader to consult that reference first. In Sec.~\ref{sec_augm_tab} a less conventional matrix form of the Butcher tableau is introduced, called the augmented Butcher tableau and the standard order conditions up to and including order four are written in the matrix form. In Sec.~\ref{sec_factor} it is demonstrated that the augmented Butcher tableau matrix factorizes into a product of matrices that have some special structure and uncommon properties. In Sec.~\ref{sec_csym} the peculiar properties exemplified in the introduction with the (4,3) and (5,4) methods of Ref.~\cite{CK1994} are examined with the newly developed tools, some symmetry properties and its effects on the order of the method are proven and the questions posed in the introduction answered. In Sec.~\ref{sec_num} numerical experiments are briefly described and practical implications discussed. 
Five new (5,4), eight (6,4) and one (8,4) schemes are numerically examined for illustrative purposes.
The conclusions are presented in Sec.~\ref{sec_concl}. Correctness of the derived analytic results is also illustrated with a Matlab script, included in \ref{sec_app_matlab}.

%% file: sec_rk.tex
Runge-Kutta methods of the form
\begin{eqnarray}
y_i&=&y_t + h\sum_{j=1}^{s}a_{ij}k_j,\label{eq_yi}\\
k_i&=&f(t+hc_i,y_i),\\
i&=&1,\dots,s,\\
y_{t+h}&=&y_t+h\sum_{i=1}^{s}b_ik_i\label{eq_yth},
\end{eqnarray}
that are called \textit{standard} or \textit{classical} here, have been successfully applied for solving initial value problems of the form
\begin{equation}
\label{eq_dydt}
\frac{dy}{dt}=f(t,y),
\end{equation}
with some initial condition $y(t_0)$. The coefficients $c_i$ are called \textit{nodes} and $b_i$ \textit{weights} of the method and only explicit Runge-Kutta methods with $a_{ij}=0$ for $i\leqslant j$ also satisfying
\begin{equation}
\label{eq_c_from_a}
c_i=\sum_{j=1}^{i-1}a_{ij}
\end{equation}
are considered here.

To minimize the memory footprint, Williamson suggested~\cite{WILLIAMSON198048} a different form of an explicit Runge-Kutta method:
\begin{eqnarray}
\Delta y_i &=& A_{i}\Delta y_{i-1}+hf(t+c_{i}h,y_{i-1}),\label{eq_2N_W_dyi}\\
y_i &=& y_{i-1} + B_{i}\Delta y_i,\\
i &=&1,\dots,s.\label{eq_2N_W_i}
\end{eqnarray}
called 2N-storage Runge-Kutta method, to distinguish from other types of low-storage methods that have been and are being developed since then.

This work focuses exclusively on 2N-storage methods of Williamson's type and builds upon recent results of Ref.~\cite{Bazavov2025a}. The main results from there, relevant to the present work, are the following relations that hold for a general 2N-storage method as proven in Ref.~\cite{Bazavov2025a}:
\begin{equation}
\label{eq_a_bc_rec}
a_{ij}=\frac{b_j}{\displaystyle\sum_{k=1}^j b_k-c_j}\left[c_i-c_j-\sum_{k=j+1}^{i-1}a_{ik}\right],
\end{equation}
and
\begin{eqnarray}
A_i &=& \frac{b_{i-1}}{b_i}\,\frac{\displaystyle \sum_{k=1}^{i-1}b_k-c_{i}}{\displaystyle \sum_{k=1}^{i-1}b_k-c_{i-1}},\label{eq_A_bc}\\
B_i &=& \frac{b_i}{\displaystyle\sum_{k=1}^i b_k-c_i}(c_{i+1}-c_i),\label{eq_B_bc},\\
i &=& 1,\dots,s.
\end{eqnarray}
Eqs.~(\ref{eq_a_bc_rec})--(\ref{eq_B_bc}) is all one needs for the further discussion, as a very different form of a 2N-storage method will be introduced starting from these equations. This is the reason the material of this paper is separated from Ref.~\cite{Bazavov2025a}.

%% file: sec_augm_tab.tex
\begin{table}[h]
\centering
\hspace{-15mm}
\[
\begin{array}{c|cccc}
0 & & & \\[1.5mm]
c_2 & a_{21} & & & \\[1.5mm]
c_3 & a_{31} & a_{32} & & \\[1.5mm]
c_4 & a_{41} & a_{42} & a_{43} & \\[1.5mm]
\hline\\[-4mm]
& b_1 & b_2 & b_3  & b_4
\end{array}
\]
\caption{The Butcher tableau of an explicit Runge-Kutta method with $s=4$ stages. The coefficients $c_i$ are called the nodes, and $b_i$ the weights.\label{tab_B}}
\end{table}

The standard form of the Butcher tableau for an explicit Runge-Kutta method is shown in Table~\ref{tab_B}. For illustration purposes the number of stages $s=4$ is used throughout the paper, but the discussion is general. It is customary to treat the coefficients of the last stage $b_i$ differently from $a_{ij}$. The order conditions are expressed in the form shown in \ref{sec_app_oc} that involves $a_{ij}$, $b_i$ and $c_i$. In 2N-storage methods all stages are equivalent and therefore for the discussion presented here it is more convenient to represent the Butcher tableau and the $c_i$ coefficients as the following $(s+1)\times(s+1)$ matrices:
\begin{equation}
\label{eq_augm_B}
A=\left(
\begin{array}{lllll}
0        & 0        & 0        & 0 & 0 \\
a_{21} & 0        & 0        & 0 & 0 \\
a_{31} & a_{32} & 0        & 0 & 0 \\
a_{41} & a_{42} & a_{43} & 0 & 0 \\
a_{51} & a_{52} & a_{53}  & a_{54} & 0
\end{array}
\right),\,\,\,\,\,\,\,\,\,\,
C=\left(
\begin{array}{lllll}
c_1        & 0        & 0        & 0 & 0 \\
0 & c_2        & 0        & 0 & 0 \\
0 & 0 & c_3        & 0 & 0 \\
0 & 0 & 0 & c_4 & 0 \\
0 & 0 & 0  & 0 & c_5
\end{array}
\right).
\end{equation}
The last row $a_{s+1,i}\equiv b_i$, $c_1\equiv0$ and $c_{s+1}\equiv1$. The matrix $C$ is diagonal, \textit{i.e.},
\begin{equation}
\label{eq_C}
C_{ij}=c_i\delta_{ij}.
\end{equation}

\begin{defi}
Representation of the Butcher tableau in Eq.~(\ref{eq_augm_B}) is called an \textit{augmented Butcher tableau}.
\end{defi}

Consider now one of the order conditions, for instance, Eq.~(\ref{eq_oc_RK_bcac}):
\begin{equation}
\sum_{i=1}^{s+1}\sum_{j=1}^{s+1} b_i c_i a_{ij} c_j=
\sum_{i=1}^{s+1}\sum_{j=1}^{s+1}
a_{s+1,i}\sum_{k=1}^{s+1}c_i\delta_{ik}a_{kj}\sum_{l=1}^{s+1}c_j\delta_{jl}=
\sum_{i=1}^{s+1}\sum_{j=1}^{s+1}\sum_{k=1}^{s+1}\sum_{l=1}^{s+1}
a_{s+1,i}C_{ik}a_{kj}C_{jl}.
\end{equation}
In other words, the order condition (\ref{eq_oc_RK_bcac}) can be represented as a product of matrices $A$ and $C$ in the corresponding order with the last step being summation of the bottom row of the result:
\begin{equation}
\sum_{l=1}^{s+1}(ACAC)_{s+1,l}=\frac{1}{8}.\label{eq_ACAC_new1}
\end{equation}

The following two matrices will prove convenient when working with the order conditions in the form such as Eq.~(\ref{eq_ACAC_new1}) (exemplified again for $s=4$):
\begin{equation}
P=\left(
\begin{array}{lllll}
0 & 0 & 0 & 0 & 1 \\
0 & 0 & 0 & 0 & 1 \\
0 & 0 & 0 & 0 & 1 \\
0 & 0 & 0 & 0 & 1 \\
0 & 0 & 0 & 0 & 1
\end{array}
\right),\,\,\,\,\,\,\,\,\,\,
Q=\left(
\begin{array}{lllll}
1 & 1 & 1 & 1 & 1 \\
0 & 0 & 0 & 0 & 0 \\
0 & 0 & 0 & 0 & 0 \\
0 & 0 & 0 & 0 & 0 \\
0 & 0 & 0 & 0 & 0
\end{array}
\right).
\end{equation}
The effect of multiplication by $P$ on the left of some matrix $M$ is that all rows of the resulting matrix $PM$ are equal to the bottom row of the matrix $M$. The application of the $Q$ matrix will be discussed later. With the help of the matrix $P$, Eq.~(\ref{eq_ACAC_new1}) can be written as
\begin{equation}
{\rm Tr}[PACAC]=\frac{1}{8},
\end{equation}
since all rows are the same and the trace sums the elements from every column.

One can easily verify that all order conditions up to and including order four listed in \ref{sec_app_oc} can be written in the form\footnote{All but one order condition at order five can also be written in the same form. The one that is different involves a tree with two symmetric branches of length two and requires element-wise rather than matrix multiplication.}:
\begin{eqnarray}
{\rm Tr}[PA] &=&1,\label{eq_oc_RK_b_P}\\
{\rm Tr}[PAC] &=&\frac{1}{2},\label{eq_oc_RK_bc_P}\\
{\rm Tr}[PACC] &=&\frac{1}{3},\label{eq_oc_RK_bc2_P}\\
{\rm Tr}[PAAC] &=&\frac{1}{6}.\label{eq_oc_RK_bac_P}\\
{\rm Tr}[PACCC] &=&\frac{1}{4},\label{eq_oc_RK_bc3_P}\\
{\rm Tr}[PACAC] &=&\frac{1}{8},\label{eq_oc_RK_bcac_P}\\
{\rm Tr}[PAACC] &=&\frac{1}{12},\label{eq_oc_RK_bac2_P}\\
{\rm Tr}[PAAAC] &=&\frac{1}{24}.\label{eq_oc_RK_ba2c_P}
\end{eqnarray}
One can replace, \textit{e.g.}, $AAA$ with $A^3$, however, the structure is more transparent when single powers of matrices are used.

%% file: sec_factor.tex
\subsection{The $d$-form of a 2N-storage Runge-Kutta method}

First, note that Eq.~(\ref{eq_a_bc_rec}) also holds for the last row of matrix $A$:
\begin{equation}
a_{s+1,j}=\frac{b_j}{\displaystyle\sum_{k=1}^j b_k-c_j}
\left[c_{s+1}-c_j-\sum_{k=j+1}^{s}a_{s+1,k}\right]=
\frac{b_j}{\displaystyle\sum_{k=1}^j b_k-c_j}
\left[1-c_j-\sum_{k=j+1}^{s}b_{k}\right]=b_j.
\end{equation}
The order condition (\ref{eq_oc_RK_b}), $c_{s+1}\equiv1$ and $a_{s+1,i}\equiv b_i$ were used to cancel the numerator and denominator.

Assuming the corresponding denominators do not vanish, let
\begin{equation}
\label{eq_dj}
d_j\equiv\frac{B_j}{c_{j+1}-c_j}=\frac{b_j}{\displaystyle\sum_{k=1}^j b_k-c_j}
\end{equation}
and
\begin{equation}
\label{eq_f_ij}
f_{ij}\equiv\left\{
\begin{array}{ll}
c_i-c_j, & i>j,\\
0, & \mbox{otherwise},
\end{array}
\right.
\end{equation}
and additionally $d_{s+1}\equiv1$. Equation~(\ref{eq_dj}) also implies that $d_1\equiv1$. For a 2N-storage Runge-Kutta method $d_i\neq0$ since $B_i\neq0$, so as long as methods with $c_{i+1}\neq c_i$ are considered, all $d_i$ are finite.
Then the $a_{ij}$ coefficients of a 2N-storage Runge-Kutta method can be represented in the following form:
\begin{equation}
\label{eq_a_d}
a_{ij}=d_j\left[f_{ij}-\sum_{k=j+1}^{i-1}a_{ik}\right]=
d_j\left[f_{ij}-\sum_{k=j+1}^{s+1}a_{ik}\right]
\end{equation}
which is simply Eq.~(\ref{eq_a_bc_rec}) written in new variables.
The sum is extended to $s+1$ since $a_{ij}=0$ for $i\leqslant j$.
\begin{defi}
Equation~(\ref{eq_a_d}) is referred to as the $d$-form of the method.
\end{defi}

\subsection{Factorization of the Butcher tableau}

The coefficients $f_{ij}$ defined by Eq.~(\ref{eq_f_ij}) form a matrix $F$ that can be represented in a compact way with the help of the following lower triangular matrix of ones:
\begin{equation}
L=\left(
\begin{array}{lllll}
1 & 0 & 0 & 0 & 0 \\
1 & 1 & 0 & 0 & 0 \\
1 & 1 & 1 & 0 & 0 \\
1 & 1 & 1 & 1 & 0 \\
1 & 1 & 1  & 1 & 1
\end{array}
\right).
\end{equation}
Then the following holds (for arbitrary $s$, here $s=4$ method is used for illustration):
\begin{equation}
\label{eq_F_CLLC}
F\equiv\left(
\begin{array}{lllll}
0 & 0 & 0 & 0 & 0 \\
f_{21} & 0 & 0 & 0 & 0 \\
f_{31} & f_{32} & 0 & 0 & 0 \\
f_{41} & f_{42} & f_{43} & 0 & 0 \\
f_{51} & f_{52} & f_{53}  & f_{54} & 0
\end{array}
\right)\equiv\left(
\begin{array}{lllll}
0 & 0 & 0 & 0 & 0 \\
c_2-c_1 & 0 & 0 & 0 & 0 \\
c_3-c_1 & c_3-c_2 & 0 & 0 & 0 \\
c_4-c_1 & c_4-c_2 & c_4-c_3 & 0 & 0 \\
c_5-c_1 & c_5-c_2 & c_5-c_3 & c_5-c_4 & 0
\end{array}
\right)=CL-LC=[C,L].
\end{equation}

Next, by construction, $d_1=B_1/(c_2-c_1)=a_{21}/c_2=1$. Nevertheless, it is convenient to keep $d_1$ as a parameter. The non-zero elements of the first column of $A$ are given by
\begin{equation}
\label{eq_a_i1}
a_{i1}=d_1\left[f_{i1}-\sum_{k=2}^{s+1}a_{ik}\right]=f_{i1}\,d_1+\sum_{k=2}^{s+1}a_{ik}\,(-d_1).
\end{equation}
Eq.~(\ref{eq_a_i1}) can be interpreted as a dot product of a row vector
\begin{equation}
(\,\,\,f_{i1}\,\,\,a_{i2}\,\,\,a_{i3}\,\,\,\dots\,\,\,)
\end{equation}
and a column vector
\begin{equation}
\label{eq_d1_vec}
(\,\,\,d_{1}\,\,\,-d_{1}\,\,\,-d_{1}\,\,\,\dots\,\,\,)^T.
\end{equation}
When the column (\ref{eq_d1_vec}) is embedded as the first column into the identity matrix, the following holds
\begin{equation}
A=
\left(
\begin{array}{lllll}
0        & 0        & 0        & 0 & 0 \\
f_{21} & 0        & 0        & 0 & 0 \\
f_{31} & a_{32} & 0        & 0 & 0 \\
f_{41} & a_{42} & a_{43} & 0 & 0 \\
f_{51} & a_{52} & a_{53}  & a_{54} & 0
\end{array}
\right)
\left(
\begin{array}{rrrrr}
d_1 & 0 & 0 & 0 & 0 \\
-d_1 & 1 & 0 & 0 & 0 \\
-d_1 & 0 & 1 & 0 & 0 \\
-d_1 & 0 & 0 & 1 & 0 \\
-d_1 & 0 & 0  & 0 & 1
\end{array}
\right)
\end{equation}
The structure of the second column of $A$ is similar
\begin{equation}
\label{eq_a_i2}
a_{i2}=d_2\left[f_{i2}-\sum_{k=2}^{s+1}a_{ik}\right]=f_{i2}\,d_2+\sum_{k=2}^{s+1}a_{ik}\,(-d_2)
\end{equation}
and therefore
\begin{equation}
A=
\left(
\begin{array}{lllll}
0        & 0        & 0        & 0 & 0 \\
f_{21} & 0        & 0        & 0 & 0 \\
f_{31} & f_{32} & 0        & 0 & 0 \\
f_{41} & f_{42} & a_{43} & 0 & 0 \\
f_{51} & f_{52} & a_{53}  & a_{54} & 0
\end{array}
\right)
\left(
\begin{array}{rrrrr}
1 & 0\phantom{_2} & 0 & 0 & 0 \\
0 & d_2 & 0 & 0 & 0 \\
0 & -d_2 & 1 & 0 & 0 \\
0 & -d_2 & 0 & 1 & 0 \\
0 & -d_2 & 0  & 0 & 1
\end{array}
\right)
\left(
\begin{array}{rrrrr}
d_1 & 0 & 0 & 0 & 0 \\
-d_1 & 1 & 0 & 0 & 0 \\
-d_1 & 0 & 1 & 0 & 0 \\
-d_1 & 0 & 0 & 1 & 0 \\
-d_1 & 0 & 0  & 0 & 1
\end{array}
\right).
\end{equation}
Continuation of this process proves the following
\begin{lem}
\label{lem1}
The augmented Butcher tableau of a 2N-storage Runge-Kutta method (in the general case, when no numerators or denominators vanish in Eq.~(\ref{eq_dj})) can be factorized into the following product: 
\begin{equation}
\label{eq_A_F_singleDs}
A=F\hat D_{s+1}\hat  D_s\dots\hat D_2\hat D_1,
\end{equation}
where the matrix $F$ is given by Eq.~(\ref{eq_F_CLLC}) and the matrices $\hat D_i$ are defined as
\begin{equation}
\label{eq_hatD_i}
(\hat D_i)_{mn}\equiv\left\{
\begin{array}{rl}
1\phantom{_i}, & m=n\neq i,\\
d_i, & m=n=i,\\
-d_i, & n=i, m>n,\\
0\phantom{_i}, & \mbox{otherwise},
\end{array}
\right.
\end{equation}
and $d_1=d_{s+1}=1$.
\end{lem}

After combining all the factors $\hat D_i$ in the Eq.~(\ref{eq_A_F_singleDs}) into a single matrix
\begin{equation}
\label{eq_D_Di}
D\equiv\prod_{i=1}^{s+1}D_{s+2-i},
\end{equation}
the factorization of the Butcher tableau is simply
\begin{equation}
\label{eq_A_FD}
A=FD.
\end{equation}
As $F$ is not invertible, factorization in Eq.~(\ref{eq_A_FD}) does not appear to be related to the Shu-Osher form~\cite{SHU1988439}.

\subsection{Explicit form and properties of the factor $D$}

The matrix $D$ defined by Eq.~(\ref{eq_D_Di}) and its inverse have some special properties to be used in the later proofs. They will be explicitly constructed and their properties proven in this and the next section.

The following products of the parameters $d_i$ will prove useful:
\begin{equation}
\label{eq_z_ij}
d_{ij}\equiv\left\{
\begin{array}{rl}
0\phantom{_i}, & i<j,\\
d_i, & i=j,\\
-d_j\left[\prod_{k=j+1}^{i-1}(1-d_k)\right]d_i, & i>j.
\end{array}
\right.
\end{equation}

\begin{lem}
\label{lem2}
For the sum, denoted $S_{il}$, of quantities $d_{ij}$, defined by Eq.~(\ref{eq_z_ij}), on the second index the following holds
\begin{equation}
\label{eq_Sr_il}
S_{il}\equiv\sum_{j=1}^l d_{ij}=\left\{
\begin{array}{rl}
0, & l\geqslant i>1,\\
-\left[\prod_{k=l+1}^{i-1}(1-d_k)\right]d_i, & l<i,\\
1, & l\geqslant i=1.
\end{array}
\right.
\end{equation}
\end{lem}
\begin{proof}
The proof is by induction. Consider first $i>l$ (and $i>1$). Then
$$
S_{i1}=d_{i1}=-d_1\left[\prod_{k=2}^{i-1}(1-d_k)\right]d_i=-\left[\prod_{k=2}^{i-1}(1-d_k)\right]d_i,
$$
since $d_1=1$. Thus, $S_{i1}$ satisfies Eq.~(\ref{eq_Sr_il}). Assume now that Eq.~(\ref{eq_Sr_il}) holds for some $l=n<i-2$. Then
\begin{eqnarray}
S_{i,n+1}&=&S_{i,n}+d_{i,n+1}=-\left[\prod_{k=n+1}^{i-1}(1-d_k)\right]d_i
-d_{n+1}\left[\prod_{k=n+2}^{i-1}(1-d_k)\right]d_i\nonumber\\
&=&-(1-d_{n+1})\left[\prod_{k=n+2}^{i-1}(1-d_k)\right]d_i
-d_{n+1}\left[\prod_{k=n+2}^{i-1}(1-d_k)\right]d_i\nonumber\\
&=&-\left[\prod_{k=(n+1)+1}^{i-1}(1-d_k)\right]d_i.
\end{eqnarray}
Thus, Eq.~(\ref{eq_Sr_il}) also holds for $l=n+1$.
For $l=i-1$
$$
S_{i,i-1}=-d_i
$$
and thus
\begin{equation}
S_{ii}=S_{i,i-1}+d_{ii}=-d_i+d_i=0.
\end{equation}
For $i<l$ the sum stops at $j=i$ since the rest of the sum includes $d_{ij}=0$ for $i<j$.

For $i=l=1$
$$
S_{11}=d_{11}=d_1=1.
$$
This completes the proof.
\end{proof}

\begin{thm}
The matrix elements of the product matrix $D$ defined in Eq.~(\ref{eq_D_Di}) are the quantities $d_{ij}$ defined in Eq.~(\ref{eq_z_ij}).
\end{thm}
\begin{proof}
Let a partial product be represented as
\begin{equation}
{\cal D}_n
\equiv\hat D_{s+1}\hat D_{s}\dots\hat D_n
\equiv\prod_{i=s+1}^{n}\hat D_{i},
\end{equation}
where the index takes the values $i=s+1,s,\dots,n$, in that order. Obviously, $D={\cal D}_1$.

When multiplication is carried from left to right, each matrix $\hat D_i$ can only modify the $i$-th column of the partial result matrix on the left. The partial product ${\cal D}_n$ must have the following form: the block with the column index $j\geqslant n$ is filled with elements $d_{ij}$, diagonal elements from $1$ up to $n-1$ are 1 and the rest are 0. In other words
\begin{equation}
({\cal D}_n)_{ij}\equiv\left\{
\begin{array}{rl}
\label{eq_calDn}
d_{ij}, & j\geqslant n,\\
1\phantom{_{ij}}, & i=j<n,\\
0\phantom{_{ij}}, & \mbox{otherwise}.
\end{array}
\right.
\end{equation}
This will be proved by induction. For $n=s+1$ this holds, the only filled block is the last column with the only non-zero element at the position $(s+1,s+1)$ which is $d_{s+1,s+1}=d_{s+1}=1$. Now it is assumed that the form holds for $n=m$ and the next $n=m-1$ is considered. Then
$$
{\cal D}_{m-1}={\cal D}_m\hat D_{m-1}.
$$
Because of the special structure of $\hat D_{m-1}$ the matrices ${\cal D}_{m-1}$ and ${\cal D}_m$ only differ in the $(m-1)$-st column. In that column, all elements in rows above $m-1$ are 0. The diagonal element $(m-1,m-1)$ is the product of a row of zeros except of 1 in the position $m-1$ and the $(m-1)$-st column of $\hat D_{m-1}$ and is thus equal to $(\hat D_{m-1})_{m-1,m-1}=d_{m-1}$. For the remaining rows $i\geqslant m$
$$
({\cal D}_{m-1})_{i,m-1}=\sum_{l=1}^{s+1}({\cal D}_m)_{il}(\hat D_{m-1})_{l,m-1}=
\sum_{l=m}^{s+1}({\cal D}_m)_{il}(-d_{m-1})=-d_{m-1}\sum_{l=m}^{s+1}d_{il}.
$$
The lower bound was raised to $l=m$ since for smaller $l$ the ${\cal D}_m$ elements are zero in those rows and the explicit form of the $(m-1)$-st column of $\hat D_{m-1}$ was also used. Next,
$$
\sum_{l=m}^{s+1}d_{il}=\sum_{l=1}^{s+1}d_{il}-\sum_{l=1}^{m-1}d_{il}=0-S_{i,m-1}.
$$
Application of Lemma~\ref{lem2}, Eq.~(\ref{eq_Sr_il}) for $i>l$ then gives
$$
({\cal D}_{m-1})_{i,m-1}=-d_{m-1}\left[\prod_{k=(m-1)+1}^{i-1}(1-d_k)\right]d_i=d_{i,m-1}.
$$
This completes the induction, as the matrix ${\cal D}_{m-1}$ has the form~(\ref{eq_calDn}). As this holds all the way down to $n=1$, this completes the proof.
\end{proof}
To visualize the structure of the matrix $D$, the $s=4$ example is shown below:
\begin{equation}
\label{eq_D_s4}
D\equiv\left(
\begin{array}{lllll}
\phantom{-}d_1 & \phantom{-}0 & \phantom{-}0 & \phantom{-}0 & 0 \\
-d_1d_2 & \phantom{-}d_2 & \phantom{-}0 & \phantom{-}0 & 0 \\
-d_1(1-d_2)d_3 & -d_2d_3 & \phantom{-}d_3 & \phantom{-}0 & 0 \\
-d_1(1-d_2)(1-d_3)d_4 & -d_2(1-d_3)d_4 & -d_3d_4 & \phantom{-}d_4 & 0 \\
-d_1(1-d_2)(1-d_3)(1-d_4)d_5 & -d_2(1-d_3)(1-d_4)d_5 & -d_3(1-d_4)d_5  & -d_4d_5 & d_5
\end{array}
\right).
\end{equation}

\begin{lem}
\label{lem3}
For the sum, denoted $V_{lj}$, of quantities $d_{ij}$, defined by Eq.~(\ref{eq_z_ij}), on the first index the following holds
\begin{equation}
\label{eq_Vc_il}
V_{lj}\equiv\sum_{i=l}^{s+1} d_{ij}=\left\{
\begin{array}{rl}
0, & l\leqslant j<s+1,\\
-d_j\left[\prod_{k=j+1}^{l-1}(1-d_k)\right], & l>j,\\
1, & l\leqslant j=s+1.
\end{array}
\right.
\end{equation}
\end{lem}
\begin{proof}
The proof is by induction starting with $s+1$ where $V_{s+1,s+1}=d_{s+1,s+1}=d_{s+1}=1$. For $j<s+1$:
$$
V_{s+1,j}=d_{s+1,j}=-d_j\left[\prod_{k=j+1}^{s}(1-d_k)\right]d_{s+1}=
-d_j\left[\prod_{k=j+1}^{s}(1-d_k)\right]
$$
since $d_{s+1}=1$. Now it is assumed that Eq.~(\ref{eq_Vc_il}) holds for $l=n>j$ and the next $V_{n-1,j}$ is considered:
\begin{eqnarray}
V_{n-1,j}&=&d_{n-1,j}+V_{nj}=-d_j\left[\prod_{k=j+1}^{n-2}(1-d_k)\right]d_{n-1}
-d_j\left[\prod_{k=j+1}^{n-1}(1-d_k)\right]\nonumber\\
&=&-d_j\left[\prod_{k=j+1}^{n-2}(1-d_k)\right]d_{n-1}
-d_j\left[\prod_{k=j+1}^{n-2}(1-d_k)\right](1-d_{n-1})\nonumber\\
&=&-d_j\left[\prod_{k=j+1}^{n-2}(1-d_k)\right].
\end{eqnarray}
For $n=j+1$, $V_{j+1,j}=-d_j$ and thus
$$
V_{jj}=d_j+V_{j+1,j}=0.
$$
For $n<j$ there is no contribution to the sum $V_{nj}$ since $d_{nj}$ are zero in this case.
\end{proof}

One can see that $S_{i,s+1}$ is the sum along the $i$-th row of $D$ and $V_{1,j}$ is the the sum along $j$-th column of $D$. Lemmas \ref{lem2} and \ref{lem3} show that $D$ has the following property: the complete sums along rows are 0 except for the first row where the sum is equal to 1, and the complete sums along columns are 0 except for the last column where the sum is equal to 1:
\begin{eqnarray}
\sum_{j=1}^{s+1}d_{ij}&=&S_{i,s+1}=\delta_{i,1},\label{eq_D_row_sum}\\
\sum_{i=1}^{s+1}d_{ij}&=&V_{1,j}=\delta_{j,s+1}.\label{eq_D_column_sum}
\end{eqnarray}

\begin{lem}
\label{lem4}
The following holds
\begin{equation}
\label{eq_DPQD}
DP=QD.
\end{equation}
\end{lem}
\begin{proof}
By construction, the resulting matrix $DP$ has only one non-zero column, the last one, $s+1$. The elements in each row are the row sums of $D$, Eq.~(\ref{eq_D_row_sum}), \textit{i.e.}, the only non-zero sum is the one in the first row and thus $DP$ is a matrix with a single non-zero element: $1$ at the position $(1,s+1)$.

Also, by construction, the matrix $QD$ has only one non-zero row, the first one. Each element is the column sum of $D$, Eq.~(\ref{eq_D_column_sum}) and the only non-zero sum is for the last column $s+1$. Thus, the $QD$ matrix has zero elements everywhere except of the one equal to 1 at the position $(1,s+1)$. Thus, Eq.~(\ref{eq_DPQD}) holds.
\end{proof}

\subsection{Properties of the inverse of the factor $D$}

By inspection, it is clear that for a given matrix $\hat D_i$ its inverse has the same structure (\textit{i.e.}, non-zero elements at the same positions) and is explicitly given by
\begin{equation}
\label{eq_invhatD_i}
(\hat D_i^{-1})_{mn}=\left\{
\begin{array}{rl}
1\phantom{_i}, & m=n\neq i,\\
1/d_i, & m=n=i,\\
1\phantom{_i}, & n=i,\,\, m>n,\\
0\phantom{_i}, & \mbox{otherwise}.
\end{array}
\right.
\end{equation}
A quick way to see this is to imagine how the inverse can be calculated with Gaussian elimination on $\hat D_i$: one will add the row $i$ to all the rows below it and then rescale the row $i$ by $d_i$. The inverse of $D$, denoted $G$, is
\begin{equation}
G\equiv D^{-1}=\prod_{i=1}^{s+1}\hat D^{-1}_i.
\end{equation}
Carrying out the multiplication shows that $G$ has the following structure
\begin{equation}
\label{eq_G_ij}
G_{ij}=
\left\{
\begin{array}{rl}
1/d_i, & i=j,\\
1\phantom{_i}, & i>j,\\
0\phantom{_i}, & \mbox{otherwise}.
\end{array}
\right.
\end{equation}
In other words, $G$ has a lower triangular structure, filled with ones similar to the matrix $L$ introduced earlier, except that the elements on the diagonal are $G_{ii}=1/d_i$.

Let the matrix $N$ be a diagonal matrix defined as
\begin{equation}
\label{eq_N}
N_{ij}=\frac{\delta_{ij}}{d_i}.
\end{equation}
Then $G$ can be represented as
\begin{equation}
G=L-I+N.
\end{equation}
For comparison, for $s=4$ the matrix $G$, the inverse of $D$ given in Eq.~(\ref{eq_D_s4}), is
\begin{equation}
\label{eq_G_s4}
G=\left(
\begin{array}{ccccc}
1/d_1 & 0 & 0 & 0 & 0 \\
1 & 1/d_2 & 0 & 0 & 0 \\
1 & 1 & 1/d_3 & 0 & 0 \\
1 & 1 & 1 & 1/d_4 & 0 \\
1 & 1 & 1  & 1 & 1/d_5
\end{array}
\right).
\end{equation}
\begin{rmk}
The matrices $C$, Eq.~(\ref{eq_C}) and $N$, Eq.~(\ref{eq_N}), completely specify the method, \textit{i.e.}, the augmented Butcher tableau matrix $A$ can be reconstructed from them.
\end{rmk}

\begin{thm}\label{th2} The following holds
\begin{equation}
\label{eq_FCG}
F=[C,G].
\end{equation}
\end{thm}
\begin{proof}
Simply put, the diagonal elements of $L$ do not matter in the definition of $F$, Eq,~(\ref{eq_F_CLLC}):
$$
[C,G]=C(L-I+N)-(L-I+N)C=[C,L]=F,
$$
since $C$ and $N$ are both diagonal and thus commute.
\end{proof}

Equation~(\ref{eq_FCG}) allows one to form other identities. First, multiplying by $D$ on the right (recall that $DG=I$):
$$
FD=CGD-GCD=C-GCG^{-1}.
$$
Recalling the factorization (\ref{eq_A_FD}) this means
\begin{equation}
\label{eq_GCG}
GCG^{-1}=C-A.
\end{equation}
Next, multiplying by $D$ both on the left and right
\begin{equation}
DFD=DCGD-DGCD=DC-CD=[D,C].
\end{equation}

%% file: sec_csym.tex
There are enough tools now to address the questions raised in Sec.\ref{sec_intro}. The situation becomes immediately clear when one examines the $d_i$ parameters of the presented pairs. They are shown in Table~\ref{tab_ds}.

\begin{table}
\centering
\begin{tabular}{r|rr|cc|cc}
$d_i$ & $(4,3)_1$ & $(4,3)_2$ & SOLUTION 1 & SOLUTION 2 &
SOLUTION 3 & SOLUTION 4 \\
\hline
$d_1$ & $1\phantom{)_2}$ & $1\phantom{)_2}$ & $1$ & $1$ & $1$ & $1$ \\
$d_2$  & $9/4\phantom{)_1}$  &  $15/4\phantom{)_2}$ &
$1.927643001997$ & $1.923666744633$ & $1.717889771931$ &
$3.668865415371$ \\
$d_3$  & $9/5\phantom{)_1}$   &  $9/5\phantom{)_2}$ &
$2.195292153589$ & $3.703493152563$ & $3.267577233123$ &
$2.081534437511$ \\
$d_4$  & $15/4\phantom{)_1}$   & $9/4\phantom{)_1}$ &
$3.703493152572$ & $2.195292153593$ & $2.081534437517$ &
$3.267577233125$ \\
$d_5$ & $1\phantom{)_2}$ & $1\phantom{)_2}$ &
$1.923666744634$ & $1.927643001997$ & $3.668865415321$ &
$1.717889771932$ \\
$d_6$ & $-\phantom{)_1}$ & $-\phantom{)_1}$ & $1$ & $1$ & $1$ & $1$ \\
\end{tabular}
\caption{The $d_i$ parameters of the two (4,3) and four (5,4) 2N-storage methods found in Ref.~\cite{CK1994}.
\label{tab_ds}
}
\end{table}

It appears that the pairs of methods are related by a transformation that reverses the order of $d_i$ parameters and replaces the $c_i$ parameters with $1-c_i$ also in the reverse order. In matrix notation this transformation can be expressed with antidiagonal transpose operation whose properties are reviewed next.

\subsection{Properties of antidiagonal transpose}

\begin{defi}
Antidiagonal unit square matrix $T$ of size $(s+1)\times(s+1)$ is defined as
\begin{equation}
T_{ij}=\delta_{i,s+2-j}=\delta_{s+2-i,j}.
\end{equation}
\end{defi}
An example of $T$ for $s=4$ is
\begin{equation}
T=\left(
\begin{array}{lllll}
0 & 0 & 0 & 0 & 1 \\
0 & 0 & 0 & 1 & 0 \\
0 & 0 & 1 & 0 & 0 \\
0 & 1 & 0 & 0 & 0 \\
1 & 0 & 0  & 0 & 0
\end{array}
\right).
\end{equation}
$T$ is its own inverse, $TT=I$, and transpose, $T^T=T$.

\begin{defi}
The antidiagonal transpose of a square matrix $M$ of size $(s+1)\times(s+1)$ is denoted as $M^\tau$ and its action on the matrix elements is defined as
\begin{equation}
(M^\tau)_{ij}=M_{s+2-j,s+2-i}.
\end{equation}
\end{defi}

\begin{lem}
\label{lem5}
The antidiagonal transpose can be expressed through regular transpose and the antidiagonal unit matrix $T$ as
\begin{equation}
M^\tau=TM^TT.
\end{equation}
\end{lem}
\begin{proof}
Using the definitions
$$
(M^\tau)_{ij}=(TM^TT)_{ij}=\sum_{k=1}^{s+1}\sum_{l=1}^{s+1}T_{ik}(M^T)_{kl}T_{lj}
=\sum_{k=1}^{s+1}\sum_{l=1}^{s+1}\delta_{s+2-i,k}M_{lk}\delta_{l,s+2-j}=M_{s+2-j,s+2-i}.
$$
\end{proof}
Clearly, $(M^\tau)^\tau=M$.

The antidiagonal transpose of the product of two matrices $A$, $B$ is
$$
(AB)^\tau=T(AB)^TT=TB^TA^TT=TB^TTTA^TT=B^\tau A^\tau,
$$
and it follows that $(A^{-1})^\tau=(A^\tau)^{-1}$.
For some of the matrices defined previously the following relations hold
$$
L^\tau=L,
$$
$$
P^\tau=TP^TT=Q.
$$

\subsection{$c$-reflection transformation}

\begin{defi}
Given an $s$-stage 2N-storage Runge-Kutta method specified by the diagonal matrices $C$ and $N$, a $c$-reflected method is defined as the one with the parameters
\begin{eqnarray}
\tilde c_i&=&1-c_{s+2-i},\label{eq_ct}\\
\tilde d_i&=&d_{s+2-i}.\label{eq_dt}
\end{eqnarray}
In matrix form the same transformations are expressed as
\begin{eqnarray}
\tilde C&=&I-C^\tau,\label{eq_ctau}\\
\tilde N&=&N^\tau.
\end{eqnarray}
\end{defi}
\begin{defi}
If $\tilde C=C$ and $\tilde N=N$ (and, as a consequence, $\tilde A=A$), the method is called self-$c$-reflected.
\end{defi}

The tilde $\tilde{}$ sign is used to denote the quantities related to the $c$-reflected method. The suggested name, $c$-reflected, originates from the fact that the nodes $c_i$ are reflected around the midpoint $c=1/2$ of the natural range $[0,1]$.
\begin{rmk}
There are Runge-Kutta methods in the literature of different type, called adjoint~\cite{HairerBook1,ButcherBook}, where the nodes $c_i$ of the adjoint method are related to the nodes of the original one in the same way as with the $c$-reflected transformation introduced here, Eq.~(\ref{eq_ct}). However, a $c$-reflected 2N-storage method is different from adjoint methods. For that reason, the suggested name $c$-reflected should avoid any confusion with the adjoint methods that are not discussed here.
\end{rmk}

The matrices $\tilde C$ and $\tilde N$ completely specify the augmented Butcher tableau $\tilde A$, and it remains to prove that $\tilde A$ indeed corresponds to a valid 2N-storage Runge-Kutta method. The fact that $\tilde A$ corresponds to a 2N-storage method is by construction: converting $\tilde c_i$ and $\tilde d_i$ to $\tilde a_{ij}$ will lead to the form (\ref{eq_a_bc_rec}), which means that $\tilde A$ corresponds to a 2N-storage method. However, its order is undetermined at this point.

One can check that the corresponding $\tilde F$ and $\tilde D$ matrices are, in fact, antidiagonal transposes of the original $F$ and $D$. First,
$$
\tilde G=L-I+N^\tau=L^\tau-I^\tau+N^\tau=G^\tau.
$$
Next, the following is taken as a \textit{definition} of $\tilde D$
$$
\tilde D\equiv(\tilde G)^{-1}=(G^\tau)^{-1}=(G^{-1})^\tau=D^\tau.
$$
And,
\begin{eqnarray}
\tilde F&\equiv&[\tilde C,L]=(I-TCT)L-L(I-TCT)=LTCT-TCTL\nonumber\\
&=&T(TLTC-CTLT)T=T(C^TT^TL^TT^T-T^TL^TT^TC^T)^TT\nonumber\\
&=&(CL^\tau-L^\tau C)^\tau=(CL-L C)^\tau=[C,L]^\tau=F^\tau.\nonumber
\end{eqnarray}
The augmented Butcher tableau of the $c$-reflected method is then
$$
\tilde A=\tilde F\tilde D=F^\tau D^\tau=(DF)^\tau.
$$
It is convenient to relate directly $A$ and $\tilde A$ with the matrix $G$ that has a particularly simple structure~(\ref{eq_G_ij})
\begin{equation}
\label{eq_tA_GAG}
\tilde A=(DF)^\tau=(DFDG)^\tau=(G^{-1}AG)^\tau=T(G^{-1}AG)^TT.
\end{equation}
It follows by construction, but can also be easily checked from the properties of $F^\tau$ and $D^\tau$ that the $c$-reflected method is self-consistent, \textit{i.e.},
\begin{equation}
\tilde c_i=\sum_{j=1}^{s+1}\tilde a_{ij}.
\end{equation}

\subsection{Order conditions for a $c$-reflected method}

There are several useful identities which will be used in the proofs. They are true both for the direct and the $c$-reflected method, as they rely on those methods having exactly the same structure. First, consider the following identity:
\begin{equation}
PC=c_{s+1}P=P.
\end{equation}
This is true because multiplication by $P$ on the left replicates the last row of the matrix on the right. For $C$ it contains $c_{s+1}=1$ in the last column. This gives the \textit{first simplifying rule}: if $C$ follows $P$ under trace, it can be dropped. Next, consider
${\rm Tr}[PMC]$,
where $M$ is some matrix. Recall that $P$ was introduced merely to be able to use the trace operation, and this expression is equivalent to summing over the bottom row, thus
\begin{equation}
\label{eq_PMC_PMA}
{\rm Tr}[PMC]=\sum_{j=1}^{s+1}\sum_{k=1}^{s+1}M_{s+1,k}C_{kj}=
\sum_{j=1}^{s+1}M_{s+1,j}c_{j}=\sum_{j=1}^{s+1}\sum_{l=1}^{s+1}M_{s+1,j}a_{jl}
={\rm Tr}[PMA].
\end{equation}
The \textit{second simplifying rule} is then the following: a $C$ in the final position in the product under trace can be exchanged for $A$ and vice versa.

The main result of this work is the following
\begin{thm}
\label{th_main}
Given an $s$-stage 2N-storage Runge-Kutta method with the augmented Butcher tableau $A$ that can be represented by matrices $C$ and $N$ (\textit{i.e.}, all $d_i$ are finite and adjacent $c_i$ are distinct) and that satisfies the order conditions for global order of accuracy $p\leqslant4$, Eqs.~(\ref{eq_oc_RK_b})--(\ref{eq_oc_RK_ba2c}),
listed also in an equivalent matrix form, Eqs.~(\ref{eq_oc_RK_b_P})--(\ref{eq_oc_RK_ba2c_P}), the $c$-reflected 2N-storage Runge-Kutta method whose augmented Butcher tableau $\tilde A$ is achieved from $A$ with the transformation given in Eq.~(\ref{eq_tA_GAG}), also satisfies the order conditions for the same order of global accuracy $p$.
\end{thm}
\begin{proof}
It remains to show that Eqs.~(\ref{eq_oc_RK_b_P})--(\ref{eq_oc_RK_ba2c_P}) hold when $A$ and $C$ are replaced by $\tilde A$ and $\tilde C$, respectively.

The proof will follow the same strategy for all relations and will be illustrated in detail for one of the order conditions first, and then the reader will be able to follow the chain of transformations for the other order conditions. A good candidate for illustrating all the steps is a condition that has some variety in positioning of the $A$ and $C$ matrices, for instance,
\begin{equation}
{\rm Tr}[P\tilde A\tilde C\tilde A\tilde C]={\rm Tr}[P\tilde A\tilde C\tilde A\tilde A].
\end{equation}
The second simplifying rule, Eq.~(\ref{eq_PMC_PMA}), is used. It is not necessary, but it shortens the proof. (Preemptively using the second simplifying rule eliminates the need to use the first simplifying rule later.) Then Eqs.~(\ref{eq_ctau}) and (\ref{eq_tA_GAG}) are used for the $c$-reflected matrices:
\begin{equation}
{\rm Tr}[PT(G^{-1}AG)^TT(I-TCT)T(G^{-1}AG)^TTT(G^{-1}AG)^TT].
\end{equation}
Some of the products $TT=I$ can be eliminated, and $C^T=C$ is always implied.
The resulting expression
\begin{equation}
\label{eq_TrP}
{\rm Tr}[PT(G^{-1}AG)^T(I-C)(G^{-1}AG)^T(G^{-1}AG)^TT]
\end{equation}
is then transformed based on the fact that trace of a matrix is the same as of its transpose. Taking the transpose of the expression under the trace in Eq.~(\ref{eq_TrP}) reverses the order, some transposes cancel and $T^T=T$ is used:
\begin{equation}
{\rm Tr}[TG^{-1}AGG^{-1}AG(I-C)G^{-1}AGTP^T].
\end{equation}
The first matrix $T$ is moved under the trace to the last position and is used to convert $TP^TT=P^\tau$ to $Q$. Some of the matrices $G$ cancel:
\begin{equation}
{\rm Tr}[G^{-1}AA(I-GCG^{-1})AGQ].
\end{equation}
The crucial steps are the following: using a) Eq.~(\ref{eq_GCG}) that follows from Theorem~\ref{th2}, and b) Lemma 4, Eq.~(\ref{eq_DPQD}) (which means that $GQG^{-1}=P$). The first $G^{-1}$ is rotated to the end and the expression becomes
\begin{equation}
{\rm Tr}[AA(I-GCG^{-1})AGQG^{-1}]=
{\rm Tr}[AA(I-(C-A))AP].
\end{equation}
Now $P$ can be rotated to the front and the expression split
\begin{equation}
{\rm Tr}[PAAA]-{\rm Tr}[PAACA]+{\rm Tr}[PAAAA].
\end{equation}
To bring the expressions to the known forms, the second simplifying rule is used again to replace an $A$ in the final position with $C$. Finally,
\begin{equation}
{\rm Tr}[P\tilde A\tilde C\tilde A\tilde C]={\rm Tr}[PAAC]-{\rm Tr}[PAACC]+{\rm Tr}[PAAAC]=\frac{1}{6}-\frac{1}{12}+\frac{1}{24}=\frac{1}{8}.
\end{equation}
Now essentially the same steps are applied to the other order conditions.
\begin{equation}
{\rm Tr}[P\tilde A]={\rm Tr}[PT(G^{-1}AG)^TT]=
{\rm Tr}[G^{-1}AGQ]={\rm Tr}[AP]={\rm Tr}[PA]=1.
\end{equation}
\begin{eqnarray}
{\rm Tr}[P\tilde A\tilde C\tilde C]&=&{\rm Tr}[P\tilde A\tilde C\tilde A]=
{\rm Tr}[PT(G^{-1}AG)^TT(I-TCT)T(G^{-1}AG)^TT]\nonumber\\
&=&{\rm Tr}[G^{-1}AG(I-C)G^{-1}AGQ]=
{\rm Tr}[A(I-(C-A))AP]\nonumber\\
&=&{\rm Tr}[PAA]-{\rm Tr}[PACA]+{\rm Tr}[PAAA]\nonumber\\
&=&{\rm Tr}[PAC]-{\rm Tr}[PACC]+{\rm Tr}[PAAC]=
\frac{1}{2}-\frac{1}{3}+\frac{1}{6}=\frac{1}{3}.
\end{eqnarray}
\begin{eqnarray}
{\rm Tr}[P\tilde A\tilde C\tilde C\tilde C]&=&{\rm Tr}[P\tilde A\tilde C\tilde C\tilde A]
={\rm Tr}[PT(G^{-1}AG)^TT(I-TCT)(I-TCT)T(G^{-1}AG)^TT]\nonumber\\
&=&{\rm Tr}[PT(G^{-1}AG)^T(I-C)(I-C)(G^{-1}AG)^TT]\nonumber\\
&=&{\rm Tr}[G^{-1}AG(I-2C+CC)G^{-1}AGQ]\nonumber\\
&=&{\rm Tr}[A(I-2GCG^{-1}+GCG^{-1}GCG^{-1})P]\nonumber\\
&=&{\rm Tr}[PA(I-2(C-A)+(C-A)(C-A))A]\nonumber\\
&=&{\rm Tr}[PA(I-2C+2A+CC-AC-CA+AA))A]\nonumber\\
&=&{\rm Tr}[PAA]-2{\rm Tr}[PACA]+2{\rm Tr}[PAAA]+{\rm Tr}[PACCA]-{\rm Tr}[PAACA]\nonumber\\
&-&{\rm Tr}[PACAA]+{\rm Tr}[PAAAA]={\rm Tr}[PAC]-2{\rm Tr}[PACC]
+2{\rm Tr}[PAAC]\nonumber\\
&+&{\rm Tr}[PACCC]-{\rm Tr}[PAACC]-{\rm Tr}[PACAC]+{\rm Tr}[PAAAC]\nonumber\\
&=&\frac{1}{2}-2\cdot \frac{1}{3}+2\cdot \frac{1}{6}+\frac{1}{4}-
\frac{1}{12}-\frac{1}{8}+\frac{1}{24}=\frac{1}{4}.
\end{eqnarray}
\begin{eqnarray}
{\rm Tr}[P\tilde A\tilde A\tilde C\tilde C]&=&{\rm Tr}[P\tilde A\tilde A\tilde C\tilde A]\nonumber\\
&=&{\rm Tr}[PT(G^{-1}AG)^TTT(G^{-1}AG)^TT(I-TCT)T(G^{-1}AG)^TT]\nonumber\\
&=&{\rm Tr}[G^{-1}AG(I-C)G^{-1}AGG^{-1}AGQ]
={\rm Tr}[PA(I-(C-A))AA]\nonumber\\
&=&{\rm Tr}[PAAA]-{\rm Tr}[PACAA]+{\rm Tr}[PAAAA]\nonumber\\
&=&{\rm Tr}[PAAC]-{\rm Tr}[PACAC]+{\rm Tr}[PAAAC]=
\frac{1}{6}-\frac{1}{8}+\frac{1}{24}=\frac{1}{12}.
\end{eqnarray}
The three remaining conditions, Eqs.~(\ref{eq_oc_RK_bc_P}), (\ref{eq_oc_RK_bac_P}) and (\ref{eq_oc_RK_ba2c_P}) can be proved in general, using the property
$$
\tilde A^n=(T(G^{-1}AG)^TT)^n=T(G^{-1}A^nG)^TT.
$$
Then
\begin{eqnarray}
\label{eq_tAnC}
{\rm Tr}[P\tilde A^n\tilde C]&=&{\rm Tr}[P\tilde A^{n+1}]=
{\rm Tr}[PT(G^{-1}A^{n+1}G)^TT]={\rm Tr}[G^{-1}A^{n+1}GQ]\nonumber\\
&=&{\rm Tr}[A^{n+1}P]={\rm Tr}[PA^{n}C].
\end{eqnarray}
\end{proof}
\begin{rmk}
Eq.~(\ref{eq_tAnC}) actually provides more information beyond proving up to the fourth order conditions. The first question is: how large $n$ can be? The structure of the matrix $A$ (and $\tilde A$) is such that each next power has one non-zero diagonal less. $A$ itself has the main diagonal being zero. $A^2$ has the first subdiagonal being zero, and so on. This means that the last power at which meaningful information can still be read off is $n=s-1$. In that case, in the last row there are two non-zero elements $(A^{s-1})_{s+1,1}$ and $(A^{s-1})_{s+1,2}$. Given that $c_1=0$, the last remaining condition is $(A^{s-1})_{s+1,2}c_2$. The proof in Eq.~(\ref{eq_tAnC}) does not rely on any other order conditions, as happens for some other third and fourth order conditions. Thus the quantities ${\rm Tr}[P A^n C]$ for $n<s$ are conserved by the $c$-reflection transformation, even for the powers of $n$ that are larger than required by the order of the method.
\end{rmk}
\begin{rmk}
All but one fifth order conditions for the $c$-reflected method can be proved in the same way as for the fourth order method presented here.
\end{rmk}

The questions posed in Sec.~\ref{sec_intro} can now be answered:
\begin{itemize}
\item[A1:] For general 2N-storage Runge-Kutta methods, where the adjacent $c_i$ parameters are different and the relations in Eq.~(\ref{eq_a_bc_rec}) hold, the augmented Butcher tableau can be factorized and reveals what is suggested to be called the $c$-reflection symmetry of 2N-storage methods: For a given general $s$-stage method of order $p\leqslant4$ (\textit{i.e.}, $c_2\neq c_3$, $c_3\neq c_4$, etc.) with the augmented Butcher tableau $A$, there is a $c$-reflected method with the augmented Butcher tableau $\tilde A$ of the same order of accuracy $p$.
\item[A2:] The transformation that relates the Butcher tableaux $A$ and $\tilde A$ of the pairs is given by Eq.~(\ref{eq_tA_GAG}).
\item[A3:] Ref.~\cite{CK1994} imposed additional conditions, namely, 
${\rm Tr}[PA^3C]=1/24$ for the third-order and ${\rm Tr}[PA^4C]=1/200$ for the fourth-order methods. As such  conditions are conserved by the $c$-reflection transformation, it became likely that the methods found were $c$-reflected transforms of each other.
\end{itemize}

\subsection{2N-storage methods of order five and above}

There is only one 2N-storage method of order five available in the literature, the one developed by Yan in Ref.~\cite{Yan2017}. Numerical evidence based on that method suggests that the following fifth-order condition, that cannot be written in the matrix form similar to Eqs.~(\ref{eq_oc_RK_b_P})--(\ref{eq_oc_RK_ba2c_P}),
\begin{equation}
\sum_{ijk}b_ia_{ij}c_ja_{ik}c_k=\frac{1}{20}
\end{equation}
breaks the $c$-reflection symmetry. This suggests that the $c$-reflection symmetry is a property of 2N-storage Runge-Kutta methods with the order of global accuracy $p\leqslant4$. Curiously, for the fifth-order 2N-storage method of Ref.~\cite{Yan2017} the breaking of the $c$-reflection symmetry is very soft, \textit{i.e.}, for the corresponding $c$-reflected method
\begin{equation}
\sum_{ijk}\tilde b_i\tilde a_{ij}\tilde c_j\tilde a_{ik}\tilde c_k\simeq0.049811.
\end{equation}
All the other order conditions at order five are satisfied. If the difference of $(0.05-0.0498)/0.05\times100\%=0.4\%$ is a specific feature of the method of Ref.~\cite{Yan2017} or a very soft breaking of the $c$-reflection symmetry at order $p=5$ is to be expected, as well as, if there are subclasses of methods at order $p\geqslant 5$ that respect the $c$-reflection symmetry, remains to be understood.

%% file: sec_num.tex
Before presenting the numerical experiments with $c$-reflected schemes, the following consideration needs to be taken into account. While analytically it is convenient to work with the transformation~(\ref{eq_tA_GAG}) relating the augmented Butcher tableau $A$ of the original and $\tilde A$ of the $c$-reflected method, that approach is susceptible to precision loss in floating point arithmetic. The effect is more pronounced for schemes with more stages. This is easy to understand, since, first, construction of the augmented Butcher tableau from the 2N-storage coefficients $A_i$, $B_i$ involves more and more multiplications at higher stages. Second, one needs to construct the matrix $D$ and its inverse $G$. After those multiply the matrix $A$ to produce $\tilde A$, the latter needs to be converted to the new 2N-storage coefficients $\tilde A_i$, $\tilde B_i$ of the $c$-reflected method. In practice, all those steps can be bypassed with the following elegant relations:
\begin{eqnarray}
A_i&=&d_{i-1}\left(\frac{1}{d_i}-1\right),\label{eq_A_d}\\
B_i&=&(c_{i+1}-c_i)\,d_i.\label{eq_B_d}
\end{eqnarray}
These formulas follow from Eq.~(\ref{eq_A_bc}), (\ref{eq_B_bc}) and (\ref{eq_dj}). Thus, in practice, one computes $d_i$ from $B_i$, then applies Eqs.~(\ref{eq_ct}), (\ref{eq_dt}) and then Eqs.~(\ref{eq_A_d}), (\ref{eq_B_d}) to convert $\tilde c_i$, $\tilde d_i$ to $\tilde A_i$, $\tilde B_i$.

The three benchmark problems used for testing the methods, the same ones as in Ref.~\cite{Bazavov2025a}, are the following:
\begin{enumerate}
\item $y'=y\cos(x)$, $y(0)=1$, $0\leqslant x\leqslant20$, with the solution $y_{exact}(x)=\exp(\sin(x))$.
\item $y'=4y\sin^3(x)\cos(x)$, $y(0)=1$, $0\leqslant x\leqslant20$, with the solution $y_{exact}(x)=\exp(\sin^4(x))$.
\item $y'=-y^{3/2}/2$, $y(0)=1$, $0\leqslant x\leqslant20$, with the solution $y_{exact}(x)=1/\sqrt{1+x}$.
\end{enumerate}
The first two are from Ref.~\cite{CK1994} and the third from Ref.~\cite{BBBproc2004}. The discussion is however on a qualitative level, as nothing special is to be expected from the $c$-reflected method in terms of its performance compared to the original one. The methods tested are the (6,4) RK46-NL method of Ref.~\cite{BERLAND20061459}, (7,4) HALERK74 of Ref.~\cite{ALLAMPALLI20093837}, (8,4) TDRKF84 of Ref.~\cite{TOULORGE20122067}, (12,4) NDBRK124 and (14,4) NDBRK144 of Ref.~\cite{NIEGEMANN2012364}. The only metric considered is the error defined as the distance between the numerical and exact solution, as function of the step size $h$:
\begin{equation}
d(h)=|y(x=20,h)-y_{exact}(x=20)|.
\end{equation}
The results for $d(h)$ with the original methods (shown in blue) and the $c$-reflected counterparts (shown in red) are plotted in Fig.~\ref{fig_tests}. 
The tests are deliberately performed in standard double precision to also investigate how different methods react to round-off errors.
First, as expected, the $c$-reflected methods scale properly as fourth-order methods. There are some nonmonotonicities, exhibited both by the original and $c$-reflected methods. Those typically arise when the error changes sign for some value of $h$ and in the vicinity of that point $d(h)$ dives to $-\infty$ on a log-log plot.
In most cases, for the three test problems the original method has smaller error, however, for the third problem the $c$-reflected TDRKF84 and NDBRK124 perform better, and the $c$-reflected NDBRK144 performs better for all three problems\footnote{This may be correlated with the fact that the original (14,4) method of Ref.~\cite{NIEGEMANN2012364} has $b_{12}\approx37.2$ and $b_{13}\approx-39.1$, while for the $c$-reflected (14,4) method $-0.1\leqslant b_i\leqslant 0.23$.}. These observations alone do not mean much, as those methods were developed for specific problems and are probably the best choices for them. The main point of this analysis is the following. The fourth-order methods compared here were found with numerical search, as no analytic solutions for fourth-order 2N-storage methods are available. When solving systems of nonlinear algebraic equations, the solutions found may depend on the initial guesses and root finding or minimization algorithms used. Whole branches of solutions can be missed easily in a multidimensional search. When a solution is found, computing its $c$-reflected counterpart costs virtually zero numerical effort. Thus it is beneficial when developing, or even when using existing 2N-storage methods, to quickly check if perhaps the $c$-reflected method has better properties for the problem at hand.
Importantly, additional constraints applied in the literature are often aimed at manipulating the stability domain and are of the form of Eq.~(\ref{eq_tAnC}) with $n$ as large as acceptable by the method. The $c$-reflected methods conserve those constraints exactly, and thus must have similar stability regions.

\begin{figure}
\centering
\includegraphics[width=0.32\textwidth]{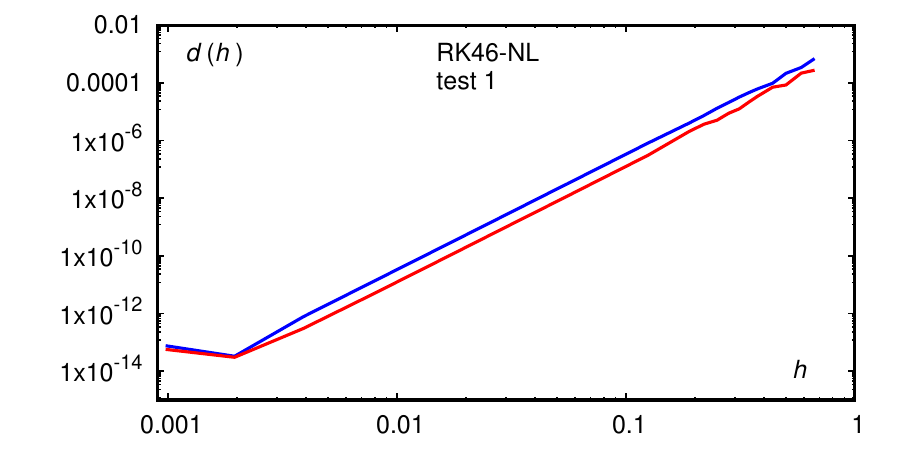}
\hfill
\includegraphics[width=0.32\textwidth]{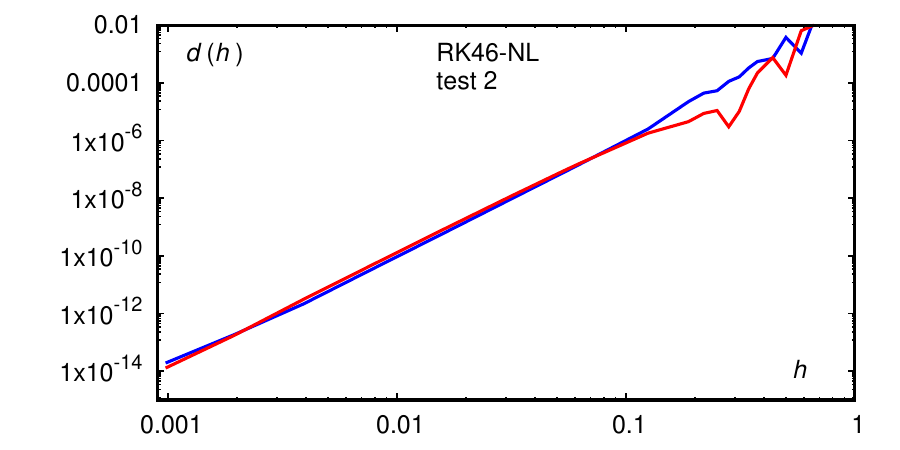}
\hfill
\includegraphics[width=0.32\textwidth]{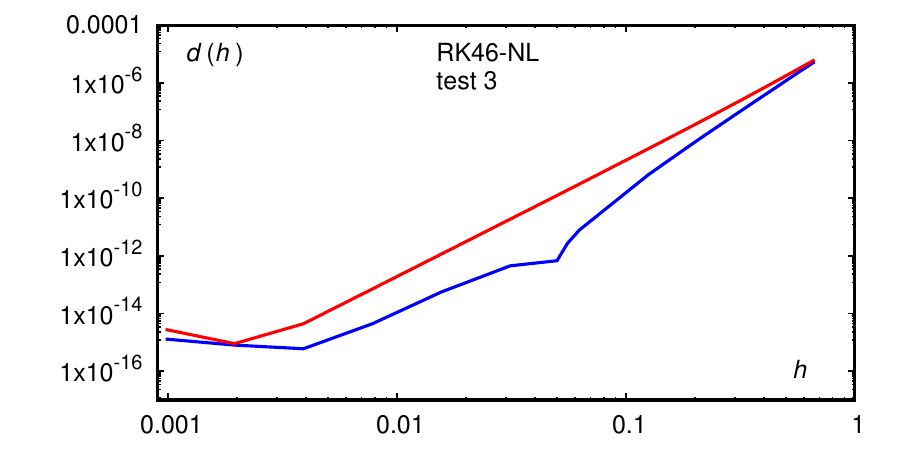}
\includegraphics[width=0.32\textwidth]{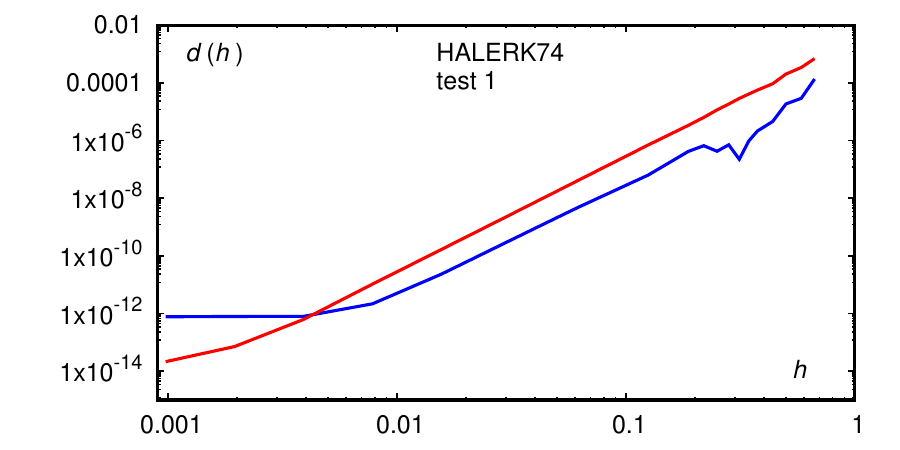}
\hfill
\includegraphics[width=0.32\textwidth]{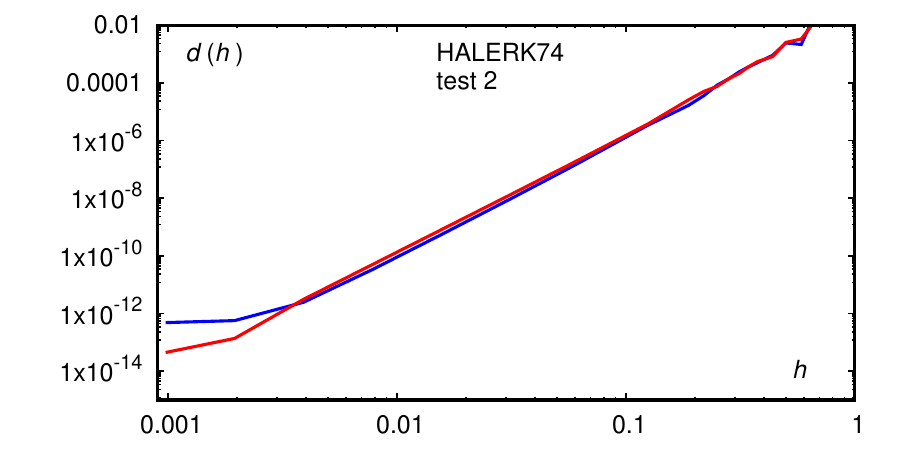}
\hfill
\includegraphics[width=0.32\textwidth]{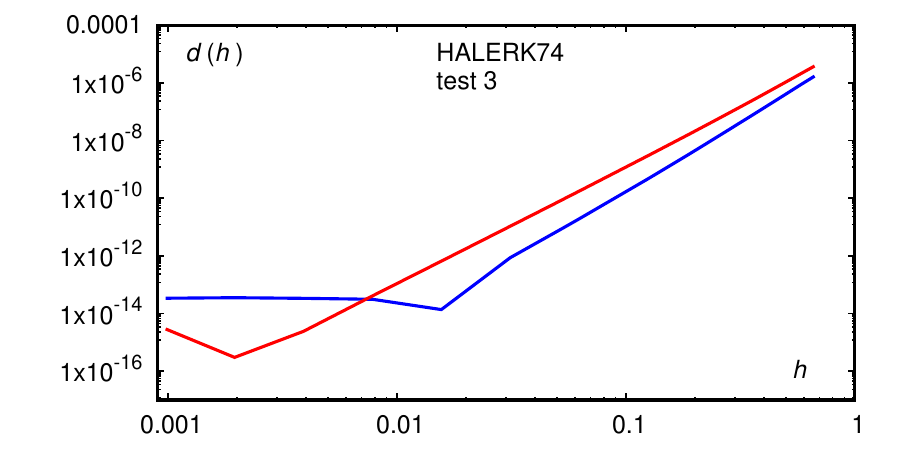}
\includegraphics[width=0.32\textwidth]{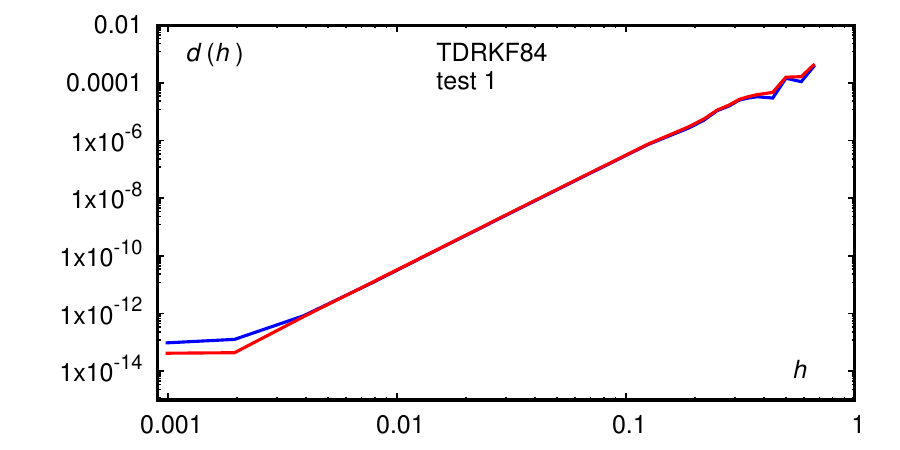}
\hfill
\includegraphics[width=0.32\textwidth]{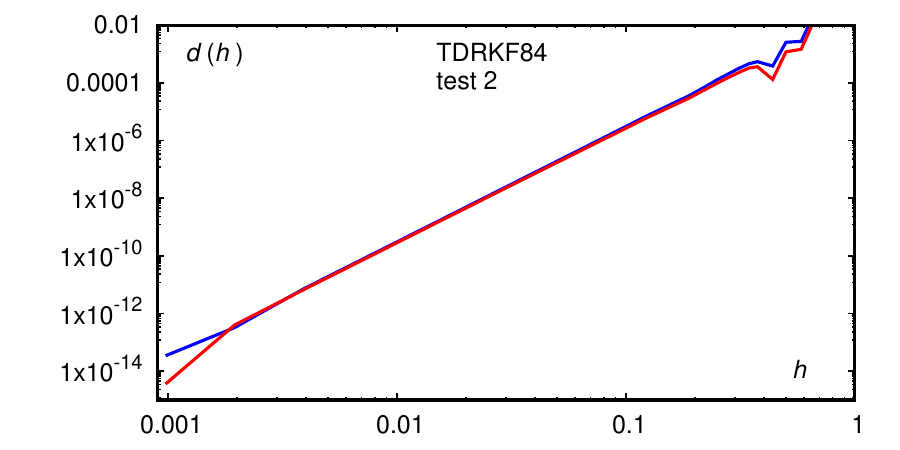}
\hfill
\includegraphics[width=0.32\textwidth]{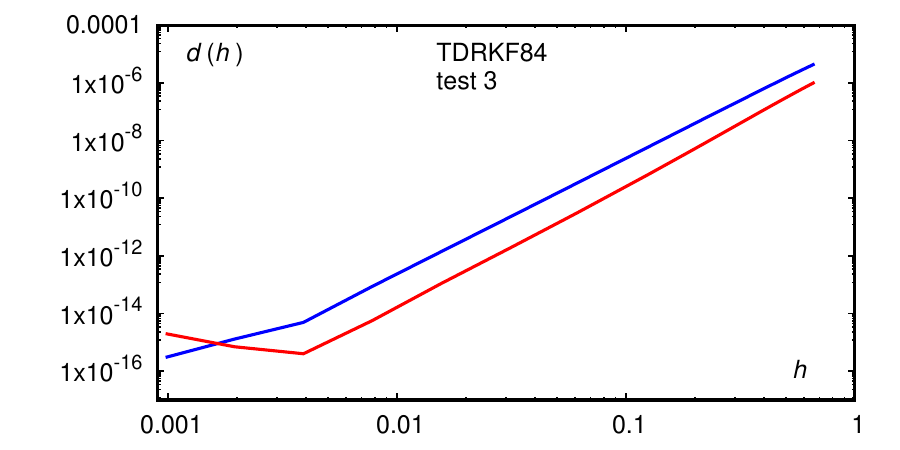}
\includegraphics[width=0.32\textwidth]{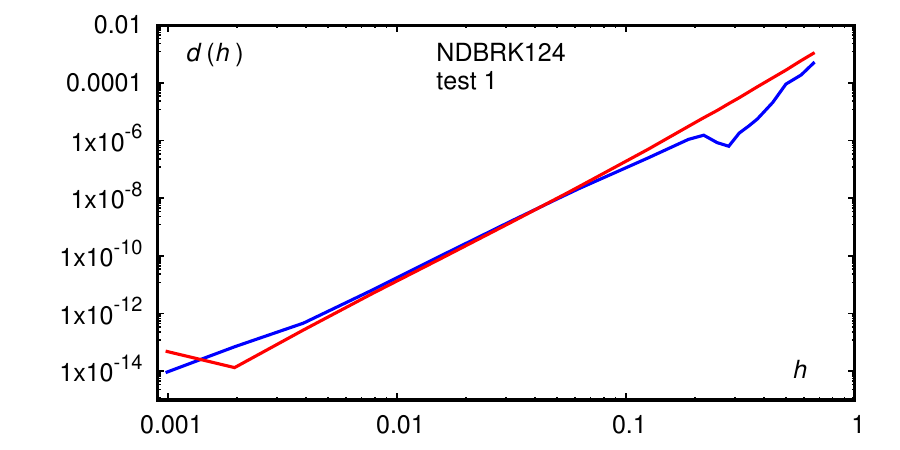}
\hfill
\includegraphics[width=0.32\textwidth]{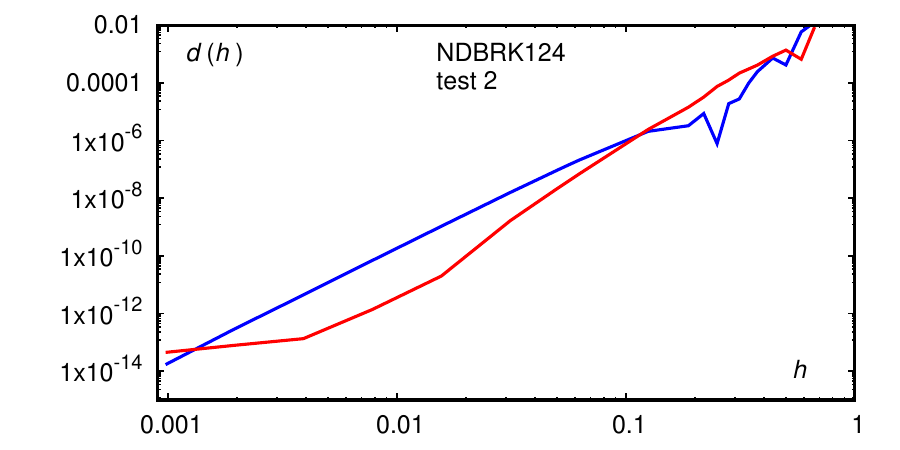}
\hfill
\includegraphics[width=0.32\textwidth]{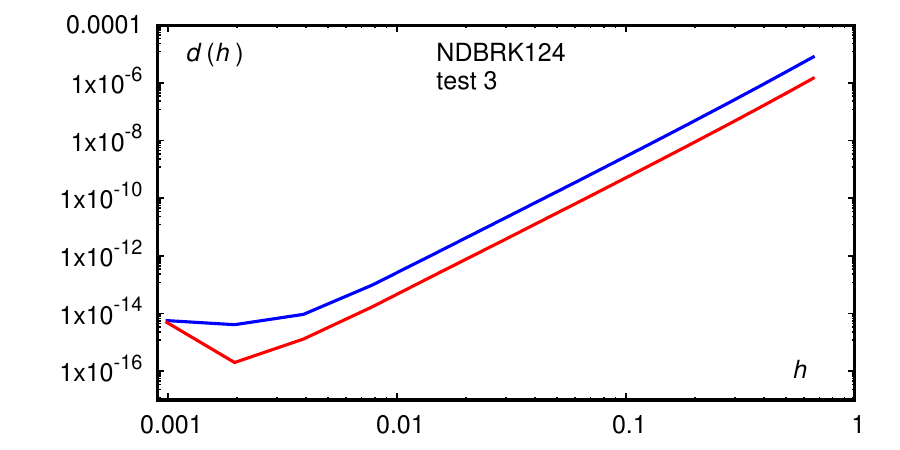}
\includegraphics[width=0.32\textwidth]{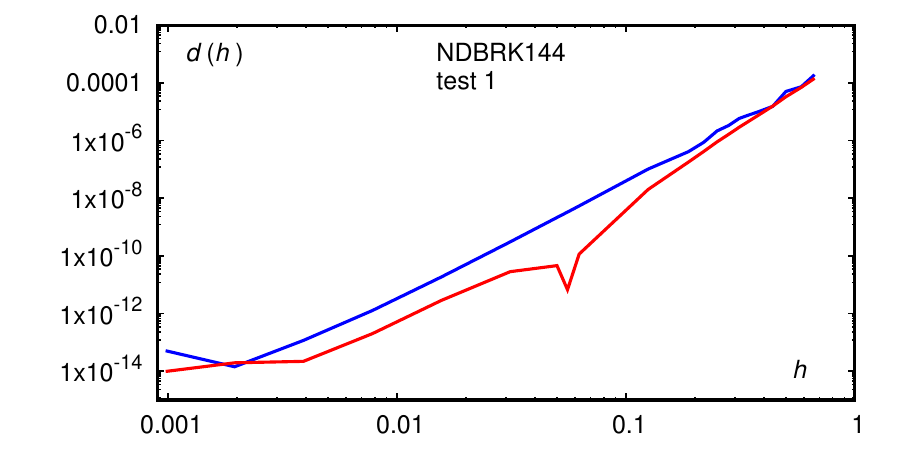}
\hfill
\includegraphics[width=0.32\textwidth]{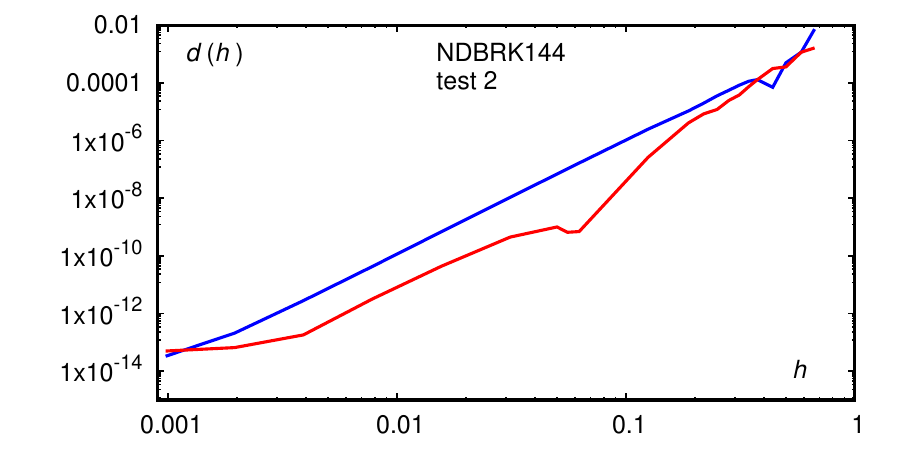}
\hfill
\includegraphics[width=0.32\textwidth]{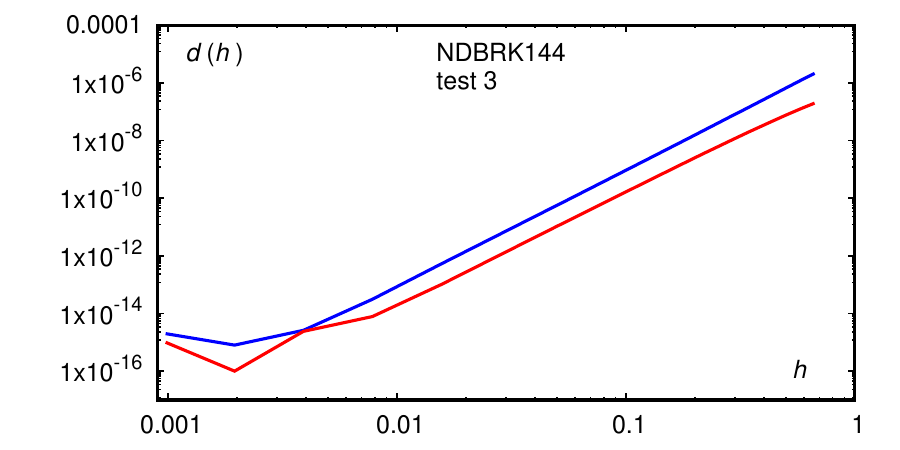}
\caption{The distance from the exact solution $d(h)$ for the three test problems (arranged horizontally) for the five fourth-order 2N-storage methods (arranged vertically), taken from literature and described in the text, shown in blue, and the corresponding $c$-reflected methods, shown in red.\label{fig_tests}
}
\end{figure}

\subsection{Properties of (5,4) methods}

It remains to discuss some features of the family of (5,4) 2N-storage Runge-Kutta methods, first explored by Carpenter and Kennedy in Ref.~\cite{CK1994}, and whose four (5,4) solutions prompted the questions in Sec.~\ref{sec_intro}. Five is the minimal number of stages required for a 2N-storage method to be of fourth order. If no additional constraints are applied, there is one free parameter, say $c_2$, similarly to how the (3,3) methods of Williamson form a one-parameter family. Williamson's classification of (3,3) methods is in the $c_2-c_3$ plane, according to Eq.~(\ref{eq_Wc3}), where most points (except, for instance, $c_2=0$, $c_2=c_3=1/3$) correspond to standard Runge-Kutta methods but only the Williamson curves represent (3,3) 2N-storage methods. If a similar reduction were done for a (5,4) method, \textit{i.e.}, the full system of nonlinear equations were brought to a triangular form, one could envision the final equation involving $c_2$ and $c_5$, where either of them could be taken as a free parameter. (Alternatively, $c_3$, $c_4$ pair could be chosen.) Conceptually, this can be solved with Groebner basis reduction, however, given the double exponential scaling of the reduction algorithms, analytically solving the constraints for (5,4) 2N-storage methods may take a while. Some insight, however, can be gained numerically.
Recently, 2N-storage constraints have been derived explicitly in the $a_{ij}$, $b_i$ variables~\cite{Bazavov2025a}, rather than $A_i$, $B_i$ variables. In that form the constraints are somewhat easier to solve numerically.
Using those relations a number of branches of solutions to the (5,4) 2N-storage conditions, \textit{i.e.}, the standard order conditions and the 2N-storage constraints, were found in the course of this work and they are presented in the $c_2-c_5$ plane in Fig.~\ref{fig_CKcurves}. The curves represent such values of $c_2$ and $c_5$ that correspond to legitimate (5,4) 2N-storage Runge-Kutta methods. Most of the solutions do not have desired properties and therefore are not considered as candidate methods for actual use. However, finding the solutions to the (5,4) constraints, even numerically, sheds some light on the structure of 2N-storage methods at fourth order. From the complexity of the plot, one can deduce that if an analog of Eq.~(\ref{eq_Wc3}) is ever derived for the (5,4) methods, it must be of high order. Several comments about Fig.~\ref{fig_CKcurves} are in order.
\begin{figure}
\centering
\includegraphics[width=0.6\textwidth]{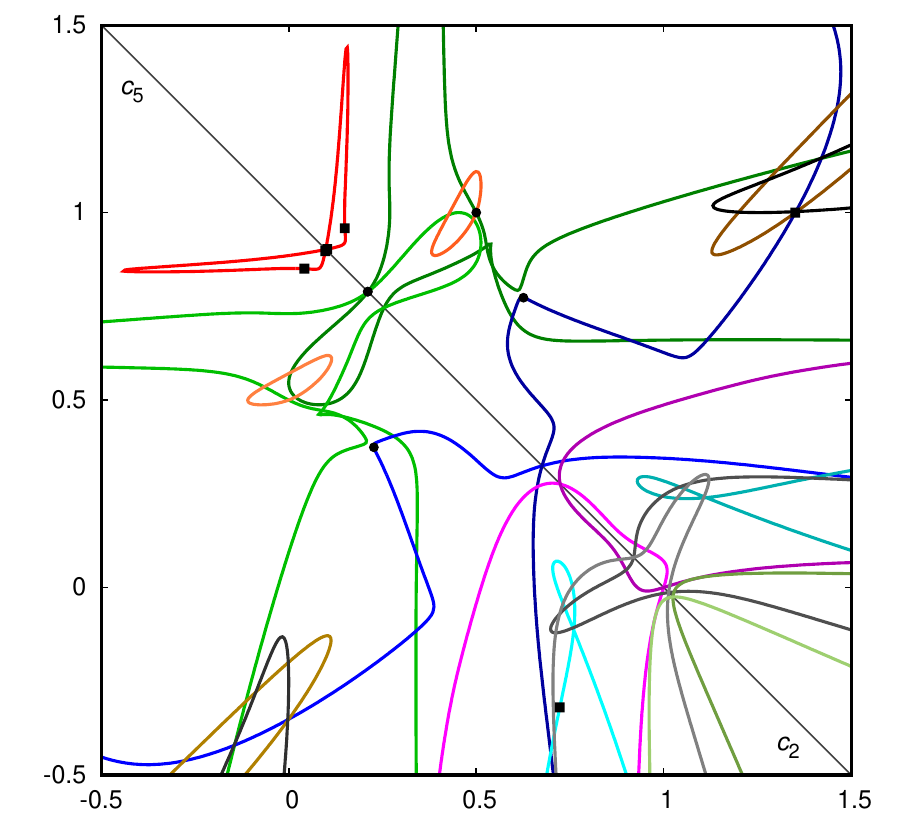}
\caption{
An analog of the Williamson curves for (3,3) methods, Fig.~\ref{fig_Wcurves}, for (5,4) methods, explored first by Carpenter and Kennedy in Ref.~\cite{CK1994}. Pairs of branches that are related by the $c$-reflection symmetry are shown in similar but slightly lighter and slightly darker pairs of colors of the same tone. The four black squares on the red branch are the four numerical solutions found in Ref.~\cite{CK1994} and the other two black squares, on the light cyan and dark blue branches, are the analytic solutions from that reference. The four black circles correspond to the five new schemes discussed in the text.\label{fig_CKcurves}
}
\end{figure}
First, it exhibits the expected $c$-reflection symmetry. Second, the four black squares on the solution branch shown in red (two of them almost fuse into one on the scale of the figure) are the four solutions of Ref.~\cite{CK1994}. Interestingly, if one prefers solutions where the nodes are in increasing order and the weights $b_i$ are not too large and negative, one is limited mostly to the red branch. Also, Ref.~\cite{CK1994} imposed an additional constraint
\begin{equation}
\label{eq_54_fifth}
\sum_{ijkl}b_ia_{ij}a_{jk}a_{kl}c_l=b_5a_{54}a_{43}a_{32}c_2=\frac{1}{200}
\end{equation}
to achieve specific stability properties for using the methods as time propagation schemes in solving partial differential equations. Somewhat luckily the red branch allows for that, however, larger values, \textit{e.g.} $1/120$ are not supported, as it appears that no solutions on the red branch have the left hand side of Eq.~(\ref{eq_54_fifth}) larger than approximately $1/175$. In other words, the window of (5,4) 2N-storage methods with increasing $c_i$ and reasonable $b_i$ values where the stability region can be optimized is quite narrow.
\begin{table}[t]
\centering
\begin{tabular}{c|c|r|c|r}
& \multicolumn{2}{|c|}{(5,4)$_1$} &
\multicolumn{2}{|c}{(5,4)$_2$} \\
\hline
$A_1$  & $0$\vphantom{$-\frac{4(1-1/\sqrt{3})}{2+\sqrt{3}-\sqrt{4/\sqrt{3}-1}}$ }  & $-$ & $0$ & $0$ \\
$A_2$ 
& $-\frac{4(1-1/\sqrt{3})}{2+\sqrt{3}-\sqrt{4/\sqrt{3}-1}}$ & $-0.653$
& $-\frac{4(1-1/\sqrt{3})}{2+\sqrt{3}+\sqrt{4/\sqrt{3}-1}}$ & $-0.347$ \\
$A_3$
& $\frac{9+\sqrt{3}(2+\sqrt{4/\sqrt{3}-1})}{15-\sqrt{3}(8+(5-2\sqrt{3})\sqrt{4/\sqrt{3}-1})}$
& $-7.601$
& $\frac{9+\sqrt{3}(2-\sqrt{4/\sqrt{3}-1})}{15-\sqrt{3}(8-(5-2\sqrt{3})\sqrt{4/\sqrt{3}-1})}$
& $2.503$ \\
$A_4$  & $-1$  & $-$\vphantom{$-\frac{4(1-1/\sqrt{3})}{2+\sqrt{3}-\sqrt{4/\sqrt{3}-1}}$ } & $-1$ & $-$ \\
$A_5$
& $\frac{15-\sqrt{3}(8+(5-2\sqrt{3})\sqrt{4/\sqrt{3}-1})}{9+\sqrt{3}(2+\sqrt{4/\sqrt{3}-1})}$
& $-0.132$
& $\frac{15-\sqrt{3}(8-(5-2\sqrt{3})\sqrt{4/\sqrt{3}-1})}{9+\sqrt{3}(2-\sqrt{4/\sqrt{3}-1})}$ & $0.400$ \\
$B_1$
& $\frac{1}{2}\left(1-\frac{1}{\sqrt{3}}\right)$  & $0.211$\vphantom{$-\frac{4(1-1/\sqrt{3})}{2+\sqrt{3}-\sqrt{4/\sqrt{3}-1}}$ }
& $\frac{1}{2}\left(1-\frac{1}{\sqrt{3}}\right)$  & $0.211$ \\
$B_2$
& $\frac{2+\sqrt{3}-\sqrt{4/\sqrt{3}-1}}{7-\sqrt{3}(2+\sqrt{4/\sqrt{3}-1})}$ & $1.665$
& $\frac{2+\sqrt{3}+\sqrt{4/\sqrt{3}-1}}{7-\sqrt{3}(2-\sqrt{4/\sqrt{3}-1})}$ & $0.884$ \\
$B_3$
& $\frac{1}{4}\left(2-\sqrt{3}+\sqrt{\frac{4}{\sqrt{3}}-1}\right)$ & $0.353$\vphantom{$-\frac{4(1-1/\sqrt{3})}{2+\sqrt{3}-\sqrt{4/\sqrt{3}-1}}$ }
& $\frac{1}{4}\left(2-\sqrt{3}-\sqrt{\frac{4}{\sqrt{3}}-1}\right)$ & $-0.219$ \\
$B_4$
& $\frac{2-\sqrt{3}+\sqrt{4/\sqrt{3}-1}}{1+\sqrt{3}(2+\sqrt{4/\sqrt{3}-1})}$ & $0.219$
& $\frac{2-\sqrt{3}-\sqrt{4/\sqrt{3}-1}}{1+\sqrt{3}(2-\sqrt{4/\sqrt{3}-1})}$ & $-0.353$ \\
$B_5$
& $\frac{1}{8}\left(2+\sqrt{3}-\sqrt{\frac{4}{\sqrt{3}}-1}\right)$ & $0.323$
& $\frac{1}{8}\left(2+\sqrt{3}+\sqrt{\frac{4}{\sqrt{3}}-1}\right)$ & $0.610$ \\
$c_1$  & $0$  & $-$\vphantom{$-\frac{4(1-1/\sqrt{3})}{2+\sqrt{3}-\sqrt{4/\sqrt{3}-1}}$ } & $0$ & $-$ \\
$c_2$
& $\frac{1}{2}\left(1-\frac{1}{\sqrt{3}}\right)$ & $0.211$\vphantom{$-\frac{4(1-1/\sqrt{3})}{2+\sqrt{3}-\sqrt{4/\sqrt{3}-1}}$ }
& $\frac{1}{2}\left(1-\frac{1}{\sqrt{3}}\right)$ & $0.211$ \\
$c_3$
& $\frac{1}{2}\left(1+\frac{1}{\sqrt{3}}\right)$ & $0.789$\vphantom{$-\frac{4(1-1/\sqrt{3})}{2+\sqrt{3}-\sqrt{4/\sqrt{3}-1}}$ }
& $\frac{1}{2}\left(1+\frac{1}{\sqrt{3}}\right)$ & $0.789$\\
$c_4$
& $\frac{1}{2}\left(1-\frac{1}{\sqrt{3}}\right)$ & $0.211$\vphantom{$-\frac{4(1-1/\sqrt{3})}{2+\sqrt{3}-\sqrt{4/\sqrt{3}-1}}$ }
& $\frac{1}{2}\left(1-\frac{1}{\sqrt{3}}\right)$ & $0.211$ \\
$c_5$
& $\frac{1}{2}\left(1+\frac{1}{\sqrt{3}}\right)$ & $0.789$\vphantom{$-\frac{4(1-1/\sqrt{3})}{2+\sqrt{3}-\sqrt{4/\sqrt{3}-1}}$ }
& $\frac{1}{2}\left(1+\frac{1}{\sqrt{3}}\right)$ & $0.789$
\end{tabular}
\caption{A pair of new (5,4) 2N-storage Runge-Kutta methods, labeled (5,4)$_1$ (left) and (5,4)$_2$ (right), related by $c$-reflection that can be expressed in radicals. For convenience, numerical values of the coefficients are also given in the decimal form, where needed.\label{tab_54_set1}
}
\end{table}
Third, the numerical solutions presented in Fig.~\ref{fig_CKcurves} were computed in the following way.
The variables used are the entries of the Butcher tableau $c_i$, $b_i$ and $a_{i>1,j}$, 15 in total.
The equations are the eight standard order conditions and six 2N-storage constraints in the form given in Ref.~\cite{Bazavov2025a}. Any variable can be fixed as a free parameter. The equations are solved with a standard Newton-Raphson algorithm with backtracking line search~\cite{NumRecBook} starting from random initial guesses and several solutions achieved with the Bertini package~\cite{BHSW06}. The precision is 1,000 bit (about 300 significant digits) due to the use of the MPFR~\cite{MPFR} library. The tolerance at which the order conditions are satisfied is $10^{-300}$. Once a solution is found, one of the parameters is perturbed by $\varepsilon=0.0001$ and a nearby solution is found with that parameter fixed at the new value. Once there are two nearby points, a step of size $\varepsilon\in[0.0001,10]$ is made in the direction of the vector equal to the difference between the two points. The procedure is repeated from the last two points. This ``walking'' procedure very quickly (\textit{i.e.}, seconds for a C++ program on a modern laptop) populates most of the branches. However, although the curves in Fig.~\ref{fig_CKcurves} are drawn as continuous, they contain a number of discontinuities: at some values of $c_2$, $c_5$ some of the coefficients diverge. Often those are the weights $b_i$: one runs to $+\infty$ and another to $-\infty$ to maintain the order condition. In that case, one simply restarts the procedure from another point, or chooses a larger $\varepsilon$ to ``jump'' over the singularity. For instance, $c_2=0$ is not allowed for 2N-storage methods of Williamson's type, therefore, all lines crossing the $y$-axis are actually discontinuous at $c_2=0$.
There are more solution branches outside of the region shown in Fig.~\ref{fig_CKcurves}.
As Fig.~\ref{fig_CKcurves} is a projection from a 15-dimensional space, different branches of solutions do not actually cross, as may appear from the figure. Also, one has to keep in mind that these results are numerical, so there is no guarantee that all solution branches have been found. The branches of solutions shown in Fig.~\ref{fig_CKcurves} were crosschecked by bringing the system of equations to a triangular form for several values of the parameters in Singular~\cite{DGPS} and solving them with Wolfram Mathematica~\cite{Mathematica}, as described in \ref{sec_app_checks54}.
\begin{table}[t]
\centering
\begin{tabular}{c|c|r|c|r}
& \multicolumn{2}{|c|}{(5,4)$_3$} &
\multicolumn{2}{|c}{(5,4)$_4$} \\
\hline
$A_1$  & $\phantom{-}0$\vphantom{$\frac{1}{4}+\frac{1}{8}\sqrt{2}-\frac{1}{24}\sqrt{12\sqrt{2}+6}$}  & $-$ & $\phantom{-}0$ & $-$ \\
$A_2$ 
& $-\frac{1}{2}$ & $-$
& $-\frac{1}{2}$ & $-$ \\
$A_3$  & $-1$  & $-$\vphantom{$\frac{1}{4}+\frac{1}{8}\sqrt{2}-\frac{1}{24}\sqrt{12\sqrt{2}+6}$} & $-1$ & $-$ \\
$A_4$  & $-1$  & $-$\vphantom{$\frac{1}{4}+\frac{1}{8}\sqrt{2}-\frac{1}{24}\sqrt{12\sqrt{2}+6}$}& $-1$ & $-$ \\
$A_5$  & $-1$  & $-$\vphantom{$\frac{1}{4}+\frac{1}{8}\sqrt{2}-\frac{1}{24}\sqrt{12\sqrt{2}+6}$} & $-1$ & $-$ \\
$B_1$
& $\frac{1}{4}+\frac{1}{8}\sqrt{2}-\frac{1}{24}\sqrt{12\sqrt{2}+6}$  & $0.227$
& $\frac{1}{4}+\frac{1}{8}\sqrt{2}+\frac{1}{24}\sqrt{12\sqrt{2}+6}$  & $0.626$ \\
$B_2$
& $\frac{1}{2}+\frac{1}{12}\sqrt{12\sqrt{2}+6}+\frac{1}{12}\sqrt{24\sqrt{2}-30}$ & $1.065$
& $\frac{1}{2}-\frac{1}{12}\sqrt{12\sqrt{2}+6}-\frac{1}{12}\sqrt{24\sqrt{2}-30}$ & $-0.065$ \\
$B_3$
& $-\frac{1}{\sqrt{2}}$ & $-0.707$\vphantom{$-\frac{4(1-1/\sqrt{3})}{2+\sqrt{3}-\sqrt{4/\sqrt{3}-1}}$ }
& $-\frac{1}{\sqrt{2}}$ & $-0.707$ \\
$B_4$
& $\frac{1}{2}-\frac{1}{12}\sqrt{12\sqrt{2}+6}-\frac{1}{12}\sqrt{24\sqrt{2}-30}$ & $-0.065$
& $\frac{1}{2}+\frac{1}{12}\sqrt{12\sqrt{2}+6}+\frac{1}{12}\sqrt{24\sqrt{2}-30}$ & $1.065$ \\
$B_5$
& $\frac{1}{2}+\frac{1}{4}\sqrt{2}+\frac{1}{12}\sqrt{12\sqrt{2}+6}$ & $1.253$
& $\frac{1}{2}+\frac{1}{4}\sqrt{2}-\frac{1}{12}\sqrt{12\sqrt{2}+6}$ & $0.454$ \\
$c_1$  & $0$  & $-$\vphantom{$\frac{1}{4}+\frac{1}{8}\sqrt{2}-\frac{1}{24}\sqrt{12\sqrt{2}+6}$} & $0$ & $-$ \\
$c_2$
& $\frac{1}{4}+\frac{1}{8}\sqrt{2}-\frac{1}{24}\sqrt{12\sqrt{2}+6}$  & $0.227$\vphantom{$-\frac{4(1-1/\sqrt{3})}{2+\sqrt{3}-\sqrt{4/\sqrt{3}-1}}$ }
& $\frac{1}{4}+\frac{1}{8}\sqrt{2}+\frac{1}{24}\sqrt{12\sqrt{2}+6}$ & $0.626$ \\
$c_3$
& $\frac{1}{2}+\frac{1}{8}\sqrt{2}+\frac{1}{24}\sqrt{24\sqrt{2}-30}$ & $0.759$\vphantom{$-\frac{4(1-1/\sqrt{3})}{2+\sqrt{3}-\sqrt{4/\sqrt{3}-1}}$ }
& $\frac{1}{2}+\frac{1}{8}\sqrt{2}-\frac{1}{24}\sqrt{24\sqrt{2}-30}$ & $0.594$\\
$c_4$
& $\frac{1}{2}-\frac{1}{8}\sqrt{2}+\frac{1}{24}\sqrt{24\sqrt{2}-30}$ & $0.406$\vphantom{$-\frac{4(1-1/\sqrt{3})}{2+\sqrt{3}-\sqrt{4/\sqrt{3}-1}}$ }
& $\frac{1}{2}-\frac{1}{8}\sqrt{2}-\frac{1}{24}\sqrt{24\sqrt{2}-30}$ & $0.241$ \\
$c_5$
& $\frac{3}{4}-\frac{1}{8}\sqrt{2}-\frac{1}{24}\sqrt{12\sqrt{2}+6}$ & $0.374$\vphantom{$-\frac{4(1-1/\sqrt{3})}{2+\sqrt{3}-\sqrt{4/\sqrt{3}-1}}$ }
& $\frac{3}{4}-\frac{1}{8}\sqrt{2}+\frac{1}{24}\sqrt{12\sqrt{2}+6}$ & $0.773$
\end{tabular}
\caption{A pair of new (5,4) 2N-storage Runge-Kutta methods, labeled (5,4)$_3$ (left) and (5,4)$_4$ (right), related by $c$-reflection that can be expressed in radicals. For convenience, numerical values of the coefficients are also given in the decimal form, where needed.\label{tab_54_set2}
}
\end{table}
\begin{table}[h]
\centering
\parbox{.6\linewidth}{
\centering
\[
\begin{array}{c|ccccc}
0 & & & & \\[1.5mm]
\frac{1}{2} & \phantom{-}\frac{1}{2} & & & & \\[1.5mm]
\frac{1}{2} & -\frac{1}{6} & \phantom{-}\frac{2}{3} & & & \\[1.5mm]
0 & -\frac{2}{3} & \phantom{-}\frac{7}{6} & -\frac{1}{2} & & \\[1.5mm]
1 & \phantom{-}\frac{13}{30} & \phantom{-}\frac{1}{15} & \phantom{-}\frac{3}{5} & -\frac{1}{10} & \\[1.5mm]
\hline\\[-4mm]
& \phantom{-}\frac{1}{4} & \phantom{-}\frac{1}{4} & \phantom{-}\frac{5}{12} & -\frac{1}{12}  & \phantom{-}\frac{1}{6}
\end{array}
\]
}
\parbox{0.15\linewidth}{
\centering
\[
\begin{array}{c|c}
\phantom{-}0 & \phantom{-}\frac{1}{2} \\[1.5mm]
-1 & \phantom{-}\frac{2}{3} \\[1.5mm]
-1 & -\frac{1}{2} \\[1.5mm]
-11 & -\frac{1}{10} \\[1.5mm]
\phantom{-}\frac{1}{10} & \phantom{-}\frac{1}{6} \\[1.5mm]
\end{array}
\]
}
\caption{The Butcher tableau and the 2N-storage coefficients for the only (5,4) method with rational coefficients known at present, labeled (5,4)$_5$.\label{tab_54_rat}}
\end{table}

Four schemes with reasonable properties that can be expressed in radicals and one scheme with rational coefficients have been found. They do not have the nodes $c_i$ in the increasing order, but at least, they are in the interval $0\leqslant c_i\leqslant 1$.
The (5,4)$_1$, (5,4)$_2$ set, shown in Table~\ref{tab_54_set1}, was found by imposing an extra condition $c_2+c_5=1$ to understand if a self-$c$-reflected, $\tilde A=A$, method exists. (There seem to be no such a (5,4) 2N-storage method.) Apart from the (5,4)$_1$, (5,4)$_2$ pair there are less useful pairs: a) with $c_2=c_4=(1+1/\sqrt{3})/2$, $c_3=c_5=(1-1/\sqrt{3})/2$, and b) with $c_2=(2+1/2^{1/3}+2^{1/3})/6$ and $c_3=\{1/2,2c_2\}$, $c_4=\{1-2c_2,1/2\}$, $c_5=1-c_2$ (\textit{i.e.}, some $c_3$ and $c_4$ outside of the $[0,1]$ range), that can also be expressed in radicals.

The (5,4)$_3$, (5,4)$_4$ set, shown in Table~\ref{tab_54_set2}, was found with different considerations that may prove useful in the future for methods of any order and number of stages. It is known that imposing $A_{i>1}=-1$ produces valid solutions. 
With this choice the Butcher tableau is completely expressed through the coefficients $b_i$, as discussed recently in Ref.~\cite{Bazavov2025a}.
The $d$-form uncovered another possibility.
One sets $A_2=-1/2$ and solves Eq.~(\ref{eq_A_d}) for $d_2$ (recall that $d_1\equiv1$). This gives $d_2=2$. Then one sets $A_3=-1$ and solves Eq.~(\ref{eq_A_d}) for $d_3$ and finds $d_3=2$. The process than self-replicates. Thus, there is a possibility of solutions with $A_2=-1/2$, $A_{i>2}=-1$ for which all non-trivial $d_i=2$, $i=2,\dots,s$. Solving, however, requires rewriting the order conditions in the $d$-form. The latter for a (5,4) method are given in \ref{sec_app_dform}. They themselves exhibit interesting symmetries which may need to be explored in the future. Note that while $c$-reflection symmetry appears to be a feature of methods of order $p\leqslant4$, factorization of the Butcher tableau and the $d$-form of the order conditions exist at any order.
Setting $d_i=2$ in Eqs.~(\ref{eq_oc_d_bc})--(\ref{eq_oc_d_ba2c}) greatly simplifies the system and the Singular package~\cite{DGPS} in one second reduces the system to four equations: one twelfth-order equation for $c_2$ and three relations linear in $c_{i>2}$. That twelfth-order equation has two double-degenerate real solutions, methods (5,4)$_3$, (5,4)$_4$ and eight complex solutions. Perhaps, other special values of the coefficients $d_i$ lead to a similar degree of simplification.  (Note that the known $A_{i>1}=-1$ case is a special case that has no $d$-form representation.) Curiously, out of all the found solutions in Fig.~\ref{fig_CKcurves}, only the (5,4)$_3$, (5,4)$_4$ pair exhibits the property $\tilde N=N$, however, $\tilde C\neq C$, and, as a result, $\tilde A\neq A$.

\begin{figure}
	\centering
	\includegraphics[width=0.49\textwidth]{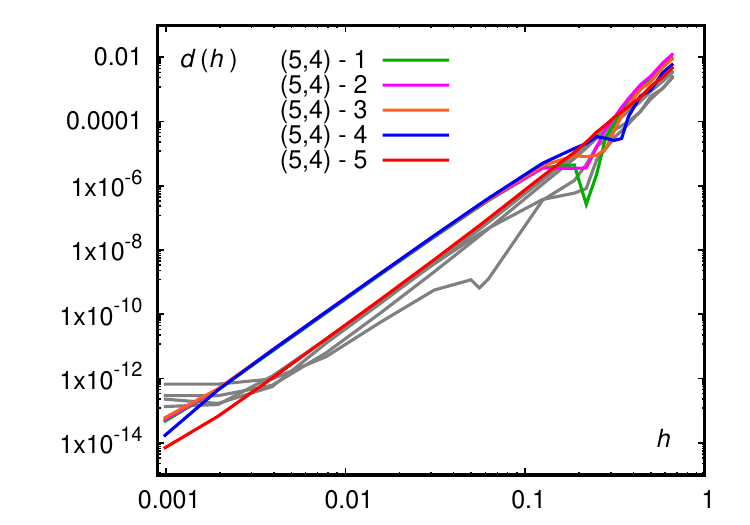}
	\hfill
	\includegraphics[width=0.49\textwidth]{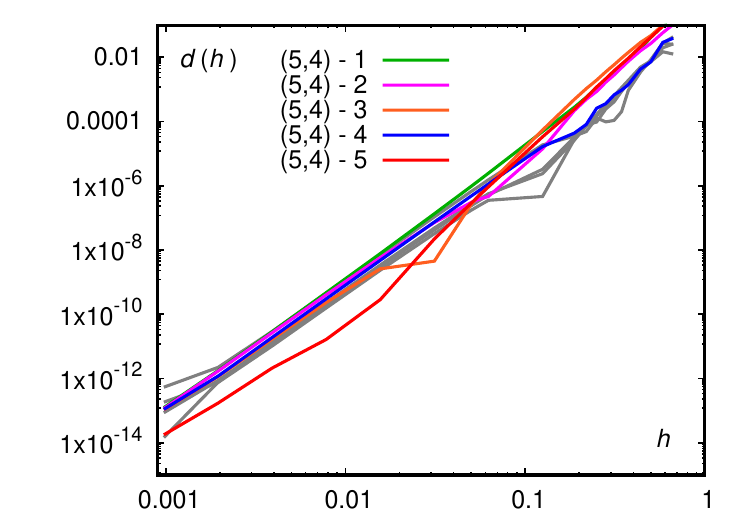}
	\caption{Scaling of the new (5,4)$_i$, $i=1,\dots5$ methods, shown in color, and the four (5,4) methods of Ref.~\cite{CK1994}, shown in gray, for test problems 1 (left) and 2 (right).\label{fig_test12}
	}
\end{figure}
By exploring special values of the coefficients $c_i$, \textit{i.e.}, $0$, $1/2$, $1$, one (5,4) 2N-storage scheme with rational coefficients, labeled (5,4)$_5$, has been found so far, shown as a black circle on the orange solution branch in Fig.~\ref{fig_CKcurves} and its Butcher tableau is listed in Table~\ref{tab_54_rat}.

The performance of the five new (5,4) schemes (in color) together with the four Carpenter-Kennedy schemes of Ref.~\cite{CK1994} (gray) for the same three test problems is shown in Figs.~\ref{fig_test12} and \ref{fig_test3}. The stability regions, defined in the standard way~\cite{ButcherBook}, are also shown in Fig.~\ref{fig_test3}. As expected, the schemes of Ref.~\cite{CK1994}, perform, in general, very well and have the largest stability region. For the first test problem, the new (5,4)$_{i=1\dots4}$ schemes perform worse than the methods of Ref.~\cite{CK1994}, while the (5,4)$_5$ scheme is competitive, Fig.~\ref{fig_test12} (left). For the second test problem the new (5,4)$_{i=1\dots4}$ schemes are competitive and the (5,4)$_5$ performs better for smaller step sizes, Fig.~\ref{fig_test12} (right).
For the third test problem all new schemes are similar in performance to the methods of Ref.~\cite{CK1994}, Fig.~\ref{fig_test3} (left). In terms of the stability regions, the new schemes are worse than the known methods, as exhibited in Fig.~\ref{fig_test3} (right). The corresponding values are
\begin{equation}
b_5a_{54}a_{43}a_{32}c_2 = \left\{
\begin{array}{ll}
\frac{1}{144}(3-\sqrt{3}), & (5,4)_{1,2},\\
\frac{1}{72}, & (5,4)_{3,4},\\
\frac{1}{360}, & (5,4)_{5}.
\end{array}
\right.\label{eq_5_extra}
\end{equation}

The (5,4)$_5$ scheme with rational coefficients illustrates some important points. $c_2=c_3=1/2$ means that this is a special case. This scheme does not have a proper $d$-form, as $d_2$ is infinite. There is also no $c$-reflected counterpart, since $c_5=1$ requires $\tilde c_2=0$, which is not allowed. Exhibited negative values of $B_i$ are typically, but not necessarily, a consequence of the nodes $c_i$ being non-monotonic.

In general, it appears that the new tools developed in Ref.~\cite{Bazavov2025a} and here can be helpful in further exploring the solution space of 2N-storage Runge-Kutta methods, as analytic results are not yet available for orders larger than three.

\begin{figure}
\centering
\includegraphics[width=0.49\textwidth]{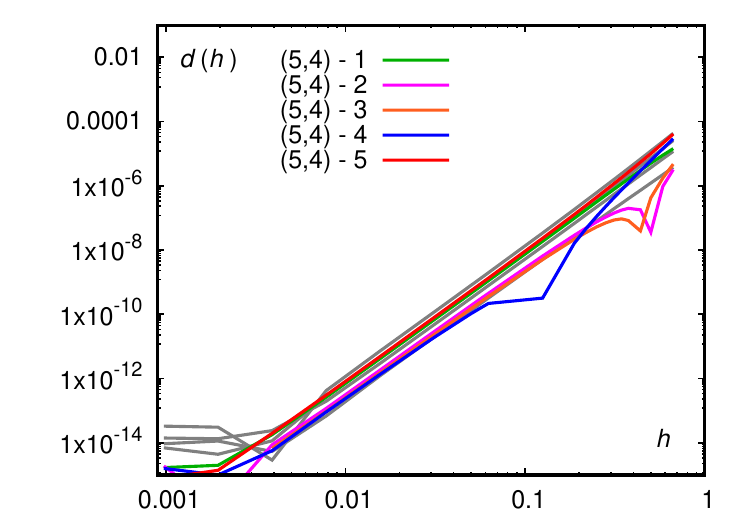}
\hfill
\includegraphics[width=0.49\textwidth]{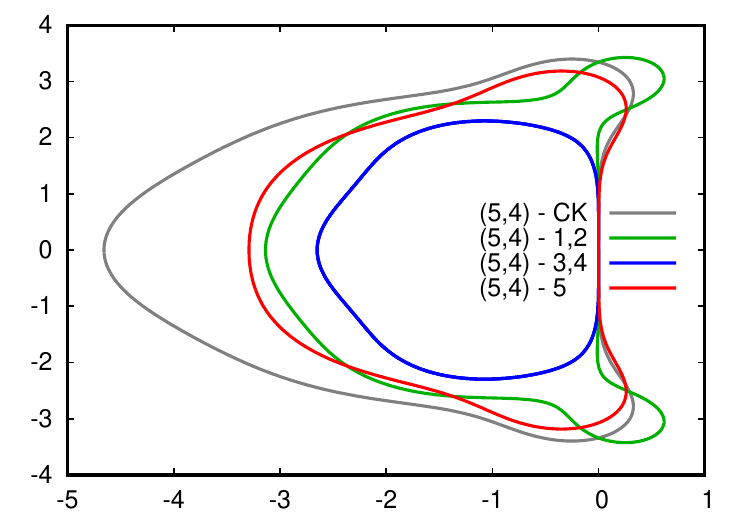}
\caption{Scaling of the new (5,4)$_i$, $i=1,\dots5$ methods, shown in color, and the four (5,4) methods of Ref.~\cite{CK1994}, shown in gray, for test problem 3 (left). Stability regions for all (5,4) methods considered (right).\label{fig_test3}
}
\end{figure}

\subsection{Experiments with (6,4) methods and beyond}

Investigating if $A_2=-1/2$, $A_{i>2}=-1$ solutions exists for (6,4) methods brings three surprises:
\begin{enumerate}
\item The $d$-form order conditions (similar to the ones listed in ~\ref{sec_app_dform} for (5,4) methods, but with all possible monomials including $c_6$ and $d_6$ added) allow for one free parameter. Note that in this case this is a system of seven equations with five variables.
\item Fixing $c_4=1/2$ produces an eighteenth-order equation for $c_2$ (after Groebner basis reduction in Singular~\cite{DGPS}) that admits solutions in radicals.
\item One of the solutions represents a self-$c$-reflected method, $\tilde A=A$. (Other solutions were not pursued due to unfavorable values of $c_2$.)
\end{enumerate}
The $c_3$ coefficient is then expressed through $c_2$ and its powers linearly, and the rest of the Butcher tableau is reconstructed from setting $c_5=1-c_3$, $c_6=1-c_2$ and $d_i=2$ for $i=2,\dots,6$ with Eq.~(\ref{eq_A_FD}). The 2N-storage coefficients can be found directly with Eqs.~(\ref{eq_A_d}) and (\ref{eq_B_d}). The solution can be expressed in the following way:
\begin{eqnarray}
\varphi_2&=&(6 \sqrt{3} + 9)^{1/3},\nonumber\\
\psi_2&=&\varphi_2 - \frac{3}{\varphi_2} + 14,\nonumber\\
\varphi_3&=&(6 \sqrt{3} - 9)^{1/3},\nonumber\\
\psi_3&=&\varphi_3 - \frac{3}{\varphi_3} + 2,\nonumber\\
c_2&=&\frac{1}{3} + \frac{1}{24}\sqrt{2\psi_2} - \frac{1}{6} \sqrt{\frac{1}{8} (42 - \psi_2) + \frac{19}{\sqrt{2 \psi_2}}},\label{eq_64_self_c2}\\
c_3&=&\frac{1}{3} + \frac{1}{24}\sqrt{2\psi_3} + \frac{1}{6} \sqrt{\frac{1}{8} (6 - \psi_3) + \frac{1}{\sqrt{2 \psi_3}}}.\label{eq_64_self_c3}
\end{eqnarray}
The Butcher tableau expressed in $c_2$ and $c_3$ parameters is given in Table~\ref{tab_64_self} and in decimals (in low precision to fit in the table, as one can always compute it to the desired precision from Eqs.~(\ref{eq_64_self_c2}) and (\ref{eq_64_self_c3})) in Table~\ref{tab_64_self_dec} together with the 2N-storage coefficients. The nodes $c_i$ are not in increasing order, but they are in the natural range $[0,1]$. The $b_4$ coefficient is negative but not too large in magnitude (the original (6,4) RK46-NL method of Ref.~\cite{BERLAND20061459} has $b_2\approx -1.27$ and $b_4\approx-1.21$ and nevertheless performs quite well on the three test problems).
The stability region is accidentally close to the one of RK46-NL as the additional ``tall tree'' constraints of order 5 and 6 are equal to $0.00802$ and $0.00108$, respectively ($0.00786$ and $0.00096$ for the RK46-NL method).

\begin{table}[h]
\centering
\[
\begin{array}{c|cccccc}
0 & & & & \\[1.5mm]
c_2 & c_2 & & & & \\[1.5mm]
c_3 & 2c_2-c_3 & 2(c_3-c_2) & & & & \\[1.5mm]
\frac{1}{2} & \frac{1}{2}-2(c_3-c_2) & 4c_3-2c_2-1 & 1-2c_3 & & & \\[1.5mm]
1-c_3 & 2c_2-c_3 & 2(c_3-c_2) & 0 & 1-2c_3 & & \\[1.5mm]
1-c_2 & c_2 & 0 & 2(c_3-c_2) & 1+2c_2-4c_3 & 2(c_3-c_2) & \\[1.5mm]
\hline\\[-4mm]
& 0 & 2c_2 & 2(c_3-2c_2) & 1-4(c_3-c_2)  & 2(c_3-2c_2) & 2c_2 
\end{array}
\]
\caption{The Butcher tableau of the self-$c$-reflected (6,4)$_1$ 2N-storage Runge-Kutta method given in terms of the parameters $c_2$ and $c_3$, Eqs.~(\ref{eq_64_self_c2}) and (\ref{eq_64_self_c3}). The 2N-storage coefficients are not tabulated due to space constraints, but they are simply: $A_1=0$, $A_2=-1/2$, $A_{i>2}=-1$ and the $B_i$ coefficients are the entries on the diagonal of the Butcher tableau. Note some symmetries: the part of the Butcher tableau above the antidiagonal filled with 0s is symmetric when mirrored with respect to the $c_4=1/2$ row, while the other part below the antidiagonal is symmetric when mirrored with respect to the $a_{i4}$ column. The $a_{i>1,j}$ block exhibits symmetry/antisymmetry with respect to the antidiagonal, depending on the parity of the subdiagonal.\label{tab_64_self}}
\end{table}

\begin{table}[h]
\centering
\parbox{.72\linewidth}{
\centering
\[
\begin{array}{r|rrrrrr}
0 & & & & \\[1.5mm]
0.1342 & 0.1342 & & & & \\[1.5mm]
0.5940 & -0.3257 & \phantom{-}0.9197 & & & & \\[1.5mm]
0.5000 & -0.4197 & 1.1077 & -0.1880& & & \\[1.5mm]
0.4060 & -0.3257 & 0.9197 & 0 & -0.1880 & & \\[1.5mm]
0.8658 & 0.1342 & 0 & 0.9197 & -1.1077 &  \phantom{-}0.9197 & \\[1.5mm]
\hline\\[-4mm]
& 0 & 0.2683 & 0.6513 & -0.8393  & 0.6513 &  \phantom{-}0.2683 
\end{array}
\]
}
\parbox{.22\linewidth}{
\[
\begin{array}{r|r}
0 & 0.1342 \\[1.5mm]
-0.5 & 0.9197 \\[1.5mm]
-1.0 & -0.1880 \\[1.5mm]
-1.0 & -0.1880 \\[1.5mm]
-1.0 &  0.9197 \\[1.5mm]
-1.0 &  0.2683 \\[1.5mm]
\end{array}
\]
}
\caption{The Butcher tableau and the 2N-storage coefficients of the self-$c$-reflected (6,4) 2N-storage Runge-Kutta method listed in Table~\ref{tab_64_self} expressed in decimals.\label{tab_64_self_dec}}
\end{table}

Keeping the self-$c$-reflected constraints, $\tilde C=C$, $\tilde N=N$, \textit{i.e.}, $c_4=1/2$, $c_5=1-c_3$, $c_6=1-c_2$, $d_6=d_2$, $d_5=d_3$, but lifting the requirement $d_{i}=2$, leads to a significantly more complicated but still solvable system in the $d$-form. Bringing it to a triangular form produces 18 real solutions (one of them being the scheme in Table~\ref{tab_64_self}) that represent self-$c$-reflected (6,4) 2N-storage schemes. Two of them have monotonic nodes $0\leqslant c_i\leqslant 1$. Unfortunately, those solutions cannot be represented in radicals. Their 2N-storage coefficients are presented in Table~\ref{tab_64_self2} with 30-digit precision. By computing their $d_i$ parameters one can verify that they are indeed self-$c$-reflected schemes.
The additional linear fifth- and sixth- order constraints for these methods are equal to $0.00554$ and $0.00071$, and $0.00613$ and $0.00079$, respectively.

\begin{table}[t]
\centering
\begin{tabular}{c|l|l}
&  \multicolumn{1}{|c|}{(6,4)$_2$} &  \multicolumn{1}{|c}{(6,4)$_3$} \\
\hline
$A_1$  & $\phantom{-}0$  & $\phantom{-}0$ \\
$A_2$  &
$-6.031817048888810491391377264767e-01$  &
$-6.416708334845571026342325707722e-01$ \\
$A_3$  &
$-1.363368319594623838259959855767e+00$  &
$-2.327050967547145814787059991623e+00$ \\
$A_4$  &
$-2.967773103326277497583509014746e-01$  &
$-2.279146815287159770076873625409e+00$ \\
$A_5$  &
$-6.263264022440050754531253706180e-01$  &
$-1.342061812908205877178026117183e+00$ \\
$A_6$  &
$-1.314148538206011636475086658292e+00$  &
$-3.862002622951114837161553551155e+00$ \\
$B_1$  &
$\phantom{-}1.709928027134245603242499940137e-01$  &
$\phantom{-}3.788384854424563341372856700833e-02$\\
$B_2$  &
$\phantom{-}4.823996417074247639531967922484e-01$  &
$\phantom{-}2.648805722810717308543538775008e-01$\\
$B_3$  &
$\phantom{-}2.997495407007149877239893157502e-01$  &
$\phantom{-}2.210064599000365587870662788315e+00$ \\
$B_4$  &
$\phantom{-}1.592788329289030209205971711461e-01$  &
$\phantom{-}5.910022949231203114708031548267e-01$ \\
$B_5$  &
$\phantom{-}4.170565624503425105376525782607e-01$  &
$\phantom{-}5.712583097229739358762356434599e-01$ \\
$B_6$  &
$\phantom{-}4.309095745334582935148984815673e-01$  &
$\phantom{-}1.057235974192264216640141659307e-01$ \\
$c_1$  & $\phantom{-}0$  & $\phantom{-}0$ \\
$c_2$  &
$\phantom{-}1.709928027134245603242499940137e-01$  &
$\phantom{-}3.788384854424563341372856700833e-02$\\
$c_3$  &
$\phantom{-}3.624178060979794860791199041947e-01$  &
$\phantom{-}1.327982832358555939494038565940e-01$\\
$c_4$  &
$\phantom{-}0.5$  & $\phantom{-}0.5$ \\
$c_5$  &
$\phantom{-}6.375821939020205139208800958053e-01$  &
$\phantom{-}8.672017167641444060505961434060e-01$\\
$c_6$  &
$\phantom{-}8.290071972865754396757500059863e-01$  &
$\phantom{-}9.621161514557543665862714329917e-01$
\end{tabular}
\caption{Two non-trivial self-$c$-reflected (6,4) 2N-storage methods whose parameters $d_i$ are different from 2.\label{tab_64_self2}
}
\end{table}

Keeping the constraint $\tilde C=C$ and lifting now the $\tilde N=N$ constraint but with $d_4=2$ fixed, produces a one-parameter family of solutions, with $c_2$ chosen as free parameter:
\begin{equation}
\label{eq_64cc_c2c3}
\left(c_2-\frac{1}{4}\right)c_3^3-\left(c_2-\frac{1}{4}\right)(c_2+1)c_3^2+
\frac{1}{4}\left[\left(c_2-\frac{1}{4}\right)\left(c_2+\frac{5}{4}\right)-\frac{1}{48}
\right]c_3-\frac{1}{24}\left(c_2-\frac{1}{4}\right)=0.
\end{equation}
The $d_2$, $d_3$, $d_5$ and $d_6$ coefficients depend linearly on the powers of $c_2$ and $c_3$. It appears that the $\tilde C=C$ with $d_4=2$ constraint manifests itself also in the relations
\begin{equation}
\frac{1}{d_2}+\frac{1}{d_6}=
\frac{1}{d_3}+\frac{1}{d_5}=1,
\end{equation}
and enforces some of the Butcher tableau symmetries of the self-$c$-reflected method shown in Table~\ref{tab_64_self}, such as zero antidiagonal, \textit{i.e.}, $b_1=a_{62}=a_{53}=0$, and the symmetry of the weights $b_2=b_6$, $b_3=b_5$. (These features have been observed numerically, but not yet proved analytically based on Eq.~(\ref{eq_64cc_c2c3}).)

\begin{table}[h]
\centering
\parbox{.46\linewidth}{
\centering
\[
\begin{array}{r|cccccc}
0 & & & & \\[1.5mm]
\frac{1}{8} & \phantom{-}\frac{1}{8} & & & & \\[1.5mm]
\frac{1}{4} & \phantom{-}\frac{5}{84} & \phantom{-}\frac{4}{21} & & & & \\[1.5mm]
\frac{1}{2} & \phantom{-}\frac{19}{42} & -\frac{20}{21} & \phantom{-}1 & & & \\[1.5mm]
\frac{3}{4} & \phantom{-}\frac{5}{84} & \phantom{-}\frac{4}{21} & \phantom{-}0 & \phantom{-}\frac{1}{2} & & \\[1.5mm]
\frac{7}{8} & \phantom{-}\frac{1}{8} & \phantom{-}0 &  \phantom{-}\frac{1}{6} & \phantom{-}\frac{5}{12} &   \phantom{-}\frac{1}{6} & \\[1.5mm]
\hline\\[-4mm]
& \phantom{-}0 & \phantom{-}\frac{4}{11} & -\frac{5}{33} & \phantom{-}\frac{19}{33}  & -\frac{5}{33} &  \phantom{-}\frac{4}{11} 
\end{array}
\]
}
\parbox{.12\linewidth}{
\[
\begin{array}{c|c}
\phantom{-}0 & \frac{1}{8} \\[1.5mm]
-\frac{11}{32} & \frac{4}{21} \\[1.5mm]
-\frac{8}{7} & 1 \\[1.5mm]
-2 & \frac{1}{2} \\[1.5mm]
-\frac{1}{2} & \frac{1}{6} \\[1.5mm]
-\frac{7}{8} &  \frac{4}{11} \\[1.5mm]
\end{array}
\]
}
\caption{The Butcher tableau and the 2N-storage coefficients of the (6,4)$_4$ scheme.\label{tab_64_cc_rat_a}}
\end{table}
A family of usable schemes exists in the window $1/8\leqslant c_2<1/6$. In that range there are solutions with $c_2<c_3<c_4=1/2$, \textit{i.e.}, all $c_i$ are monotonic. Solutions with rational coefficients exist for $c_2=1/8$, $1/4$ and possibly other values. $c_2=1/6$ and $c_2=1$ are special cases that need to be solved separately. The $c_2=1/4$ solution produces $c_3=0$ and is not pursued. A pair of $c$-reflected methods corresponding to $c_2=1/8$ is listed in Tables~\ref{tab_64_cc_rat_a} and \ref{tab_64_cc_rat_b}, labeled (6,4)$_4$ and (6,4)$_5$, respectively. The fifth- and sixth-order linear conditions for them are $4/693$ and $1/1386$. Another usable solution $c_2=1/8$, $c_3=(21 - \sqrt{57})/48\approx0.28$ produces a pair of $c$-reflected methods with monotonic $c_i$ that can be expressed in radicals. Their performance is similar to the methods in Tables~\ref{tab_64_cc_rat_a} and \ref{tab_64_cc_rat_b} and therefore their Butcher tableaux are not presented.
\begin{table}[h]
\centering
\parbox{.46\linewidth}{
\centering
\[
\begin{array}{r|cccccc}
0 & & & & \\[1.5mm]
\frac{1}{8} & \phantom{-}\frac{1}{8} & & & & \\[1.5mm]
\frac{1}{4} & -\frac{5}{44} & \phantom{-}\frac{4}{11} & & & & \\[1.5mm]
\frac{1}{2} & \phantom{-}\frac{1}{22} & \phantom{-}\frac{4}{33} & \phantom{-}\frac{1}{3} & & & \\[1.5mm]
\frac{3}{4} & -\frac{5}{44} & \phantom{-}\frac{4}{11} & \phantom{-}0 & \phantom{-}\frac{1}{2} & & \\[1.5mm]
\frac{7}{8} & \phantom{-}\frac{1}{8} & \phantom{-}0 &  \phantom{-}\frac{1}{2} & -\frac{1}{4} &   \phantom{-}\frac{1}{2} & \\[1.5mm]
\hline\\[-4mm]
& \phantom{-}0 & \phantom{-}\frac{4}{21} & \phantom{-}\frac{5}{21} & \phantom{-}\frac{1}{7}  & \phantom{-}\frac{5}{21} &  \phantom{-}\frac{4}{21} 
\end{array}
\]
}
\parbox{.12\linewidth}{
\[
\begin{array}{c|c}
\phantom{-}0 & \frac{1}{8} \\[1.5mm]
-\frac{21}{32} & \frac{4}{11} \\[1.5mm]
-\frac{8}{11} & \frac{1}{3} \\[1.5mm]
-\frac{2}{3} & \frac{1}{2} \\[1.5mm]
-\frac{3}{2} & \frac{1}{2} \\[1.5mm]
-\frac{11}{8} &  \frac{4}{21} \\[1.5mm]
\end{array}
\]
}
\caption{The Butcher tableau and the 2N-storage coefficients of the (6,4)$_5$ scheme.\label{tab_64_cc_rat_b}}
\end{table}

\begin{table}[h]
\centering
\parbox{.46\linewidth}{
\centering
\[
\begin{array}{r|cccccc}
0 & & & & \\[1.5mm]
\frac{1}{6} & \phantom{-}\frac{1}{6} & & & & \\[1.5mm]
\frac{1}{6} & -\frac{5}{24} & \phantom{-}\frac{3}{8} & & & & \\[1.5mm]
\frac{1}{2} & \phantom{-}\frac{1}{8} & \phantom{-}\frac{1}{24} & \phantom{-}\frac{1}{3} & & & \\[1.5mm]
\frac{5}{6} & -\frac{5}{24} & \phantom{-}\frac{3}{8} & \phantom{-}0 & \phantom{-}\frac{2}{3} & & \\[1.5mm]
\frac{5}{6} & \phantom{-}\frac{1}{6} & \phantom{-}0 &  \phantom{-}\frac{3}{8} & -\frac{1}{12} &   \phantom{-}\frac{3}{8} & \\[1.5mm]
\hline\\[-4mm]
& \phantom{-}0 & \phantom{-}\frac{1}{6} & \phantom{-}\frac{5}{24} & \phantom{-}\frac{1}{4}  & \phantom{-}\frac{5}{24} &  \phantom{-}\frac{1}{6} 
\end{array}
\]
}
\parbox{.12\linewidth}{
\[
\begin{array}{c|c}
\phantom{-}0 & \frac{1}{6} \\[1.5mm]
-1 & \frac{3}{8} \\[1.5mm]
-1 & \frac{1}{3} \\[1.5mm]
-\frac{1}{2} & \frac{2}{3} \\[1.5mm]
-2 & \frac{3}{8} \\[1.5mm]
-1 &  \frac{1}{6} \\[1.5mm]
\end{array}
\]
}
\caption{The Butcher tableau and the 2N-storage coefficients of the (6,4)$_6$ scheme.\label{tab_64_rat1}}
\end{table}
\begin{table}[h]
\centering
\parbox{.48\linewidth}{
\centering
\[
\begin{array}{r|cccccc}
0 & & & & \\[1.5mm]
\frac{1}{2} & \phantom{-}\frac{1}{6} & & & & \\[1.5mm]
\frac{1}{2} & -\frac{1}{6} & \phantom{-}\frac{2}{3} & & & & \\[1.5mm]
0 & -\frac{2}{3} & \phantom{-}\frac{7}{6} & -\frac{1}{2} & & & \\[1.5mm]
0 & -\frac{7}{12} & \phantom{-}\frac{13}{12} & -\frac{5}{12} & -\frac{1}{12} & & \\[1.5mm]
1 & \phantom{-}\frac{5}{12} & \phantom{-}\frac{1}{12} &  \phantom{-}\frac{7}{12} & -\frac{13}{12} &   \phantom{-}1 & \\[1.5mm]
\hline\\[-4mm]
& \phantom{-}\frac{1}{4} & \phantom{-}\frac{1}{4} & \phantom{-}\frac{5}{12} & -\frac{11}{12}  & \phantom{-}\frac{5}{6} &  \phantom{-}\frac{1}{6} 
\end{array}
\]
}
\parbox{.12\linewidth}{
\[
\begin{array}{c|c}
\phantom{-}0 & \phantom{-}\frac{1}{2} \\[1.5mm]
-1 & \phantom{-}\frac{2}{3} \\[1.5mm]
-1 & -\frac{1}{2} \\[1.5mm]
-1 & -\frac{1}{12} \\[1.5mm]
-1 & \phantom{-}1 \\[1.5mm]
-1 &  \phantom{-}\frac{1}{6} \\[1.5mm]
\end{array}
\]
}
\caption{The Butcher tableau and the 2N-storage coefficients for the  (6,4)$_7$ scheme.\label{tab_64_rat2}}
\end{table}
\begin{table}[h]
\centering
\parbox{.50\linewidth}{
\centering
\[
\begin{array}{r|cccccc}
0 & & & & \\[1.5mm]
1 & \phantom{-}1 & & & & \\[1.5mm]
1 & \phantom{-}\frac{37}{36} & -\frac{1}{36} & & & & \\[1.5mm]
\frac{1}{2} & \phantom{-}\frac{19}{36} & \phantom{-}\frac{17}{36} & -\frac{1}{2} & & & \\[1.5mm]
\frac{1}{2} & -\frac{5}{36} & \phantom{-}\frac{41}{36} & -\frac{7}{6} & \phantom{-}\frac{2}{3} & & \\[1.5mm]
1 & \phantom{-}\frac{13}{36} & \phantom{-}\frac{23}{36} &  -\frac{2}{3} & -\frac{1}{6} &   \phantom{-}\frac{1}{2} & \\[1.5mm]
\hline\\[-4mm]
& \phantom{-}\frac{1}{6} & \phantom{-}\frac{5}{6} & -\frac{31}{36} & \phantom{-}\frac{13}{36}  & \phantom{-}\frac{11}{36} &  \phantom{-}\frac{7}{36} 
\end{array}
\]
}
\parbox{.12\linewidth}{
\[
\begin{array}{c|c}
\phantom{-}0 & \phantom{-}1 \\[1.5mm]
-1 & -\frac{1}{36} \\[1.5mm]
-1 & -\frac{1}{2} \\[1.5mm]
-1 & \phantom{-}\frac{2}{3} \\[1.5mm]
-1 & \phantom{-}\frac{1}{2} \\[1.5mm]
-1 &  \phantom{-}\frac{7}{36} \\[1.5mm]
\end{array}
\]
}
\caption{The Butcher tableau and the 2N-storage coefficients for the (6,4)$_8$ scheme.\label{tab_64_rat3}}
\end{table}

Exploring now a special case, $c_4=1/2$, $c_2=c_3$, $c_5=c_6$, which has no proper $d$-form representation, produces several schemes. One that may be of interest is a scheme with rational coefficients shown in Table~\ref{tab_64_rat1} and labeled (6,4)$_6$. While its $d_2$ and $d_5$ are infinite and no $c$-reflection transformation can be defined, it has a similar structure and some of the symmetries of the self-$c$-reflected scheme in Table~\ref{tab_64_self}. The fifth- and sixth-order linear constraints are equal to $1/192$ and $1/1152$, respectively.

The commonly used $A_{i>1}=-1$ special case (that enforces $c_2=c_3$, $c_4=c_5$ and $c_6=1$~\cite{Bazavov2025a}) results in a one-parameter family of (6,4) methods, that can be parameterized by, \textit{e.g.}, $c_2$. The special points $c_2=1/3$, $1/2$ and $1$ need to be solved separately. There is not much room for tuning though, as there is a very narrow window of $c_2$ values where $c_4$ falls in $[0,1]$ range. The two rational schemes are possible for $c_2=1/2$ and $c_2=1$, as shown Tables~\ref{tab_64_rat2}, labeled (6,4)$_7$, and \ref{tab_64_rat3}, labeled (6,4)$_8$, respectively. There is no (real) solution for $c_2=1/3$. There is some resemblance between the (6,4)$_7$ scheme in Table~\ref{tab_64_rat2} and the rational (5,4)$_5$ scheme in Table~\ref{tab_54_rat}. The fifth- and sixth- order linear constraints for these two schemes are, respectively, $1/72$ and $1/432$, and  $1/216$ and $7/7776$.

The performance of the new (6,4)$_i$, $i=1,\dots,8$ schemes considered here is compared with the HALERK64 of Ref.~\cite{ALLAMPALLI20093837} and RK46-NL of Ref.~\cite{BERLAND20061459} in Figs.~\ref{fig_64test12} and \ref{fig_64test3} (left) for the three test problem. The stability regions are shown in Fig.~\ref{fig_64test3} (right). The performance of the new schemes is, in general, comparable to HALERK64 and RK46-NL, but depends on the problem. For instance, while the (6,4)$_4$ and (6,4)$_5$ schemes perform very well in problems 1 and 2, they are noticeably less accurate for problem 3. The stability domain of (6,4)$_6$ scheme is comparable to HALERK64 and of (6,4)$_1$ to RK46-NL.

Brief experiments with (8,4) schemes show that with setting all $d_i=2$, $i=2,\dots,8$ and $c_5=1/2$ again allows for building self-$c$-reflected schemes whose Butcher tableau exhibits the same symmetries as the ones in Table~\ref{tab_64_self}. With eight stages there is a one-parameter family of solutions even after $c_5=1/2$ is fixed. It appears, they share the same feature with the (6,4) self-$c$-reflected scheme with $d_i=2$: the nodes $c_i$ are non-monotonic and, at least, one of the $c_{i=2,3,4}$ has to be larger than $c_5=1/2$. Taking that into account and to fix the free parameter, an extra constraint $c_3+c_4=1$ can be added. This results in a self-$c$-reflected scheme that can also be represented in radicals. The solution is:
\begin{eqnarray}
c_2&=&\frac{1}{2}-\frac{1}{4}\sqrt{2}~\approx~0.146,\nonumber\\
c_3&=&\frac{1}{2} - \frac{1}{4}\left(\sqrt{2} - \frac{4}{3}\right)^{\frac{1}{3}}
\approx~0.392.\nonumber
\end{eqnarray}
The interested reader can reconstruct the full scheme from those two values.

It appears that schemes with an odd number of stages $s$ are structurally different from the schemes with even $s$, when $d_i=2$, $i=2,\dots,s$, although solutions seem to exist for $s\geqslant7$. Existence and properties of self-$c$-reflected schemes with $d_i\neq2$ require further investigation.

\begin{figure}[h]
\centering
\includegraphics[width=0.49\textwidth]{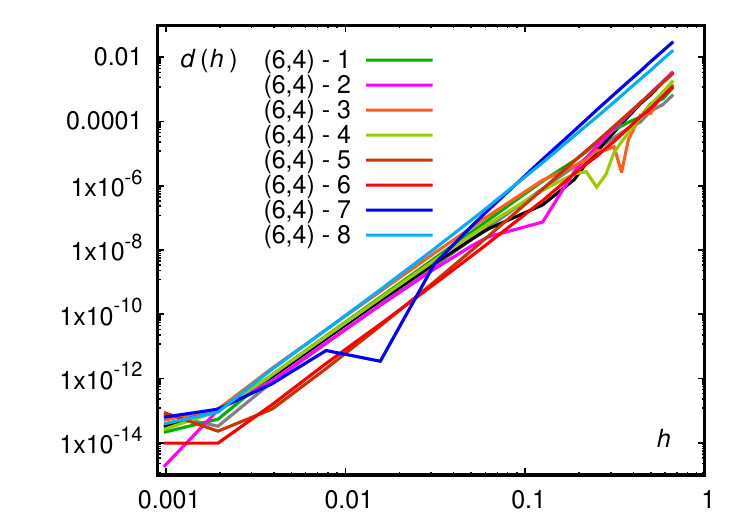}
\hfill
\includegraphics[width=0.49\textwidth]{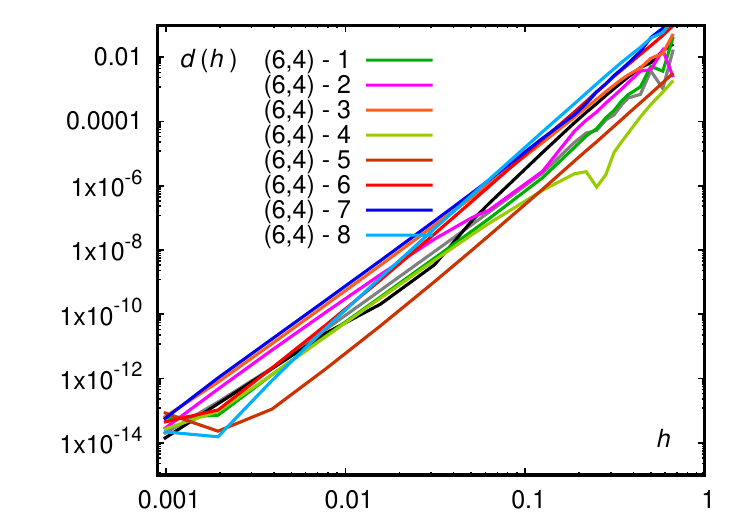}
\caption{Scaling of the new (6,4)$_i$, $i=1,\dots8$ methods, shown in color, and the HALERK64 method of Ref.~\cite{ALLAMPALLI20093837}, shown in black and the RK46-NL method of Ref.~\cite{BERLAND20061459}, shown in gray, for test problems 1 (left) and 2 (right).\label{fig_64test12}
}
\end{figure}
\begin{figure}
\centering
\includegraphics[width=0.49\textwidth]{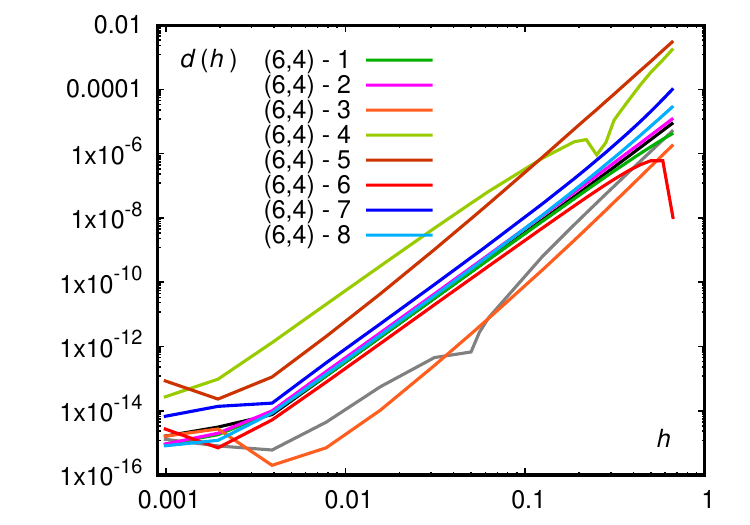}
\hfill
\includegraphics[width=0.49\textwidth]{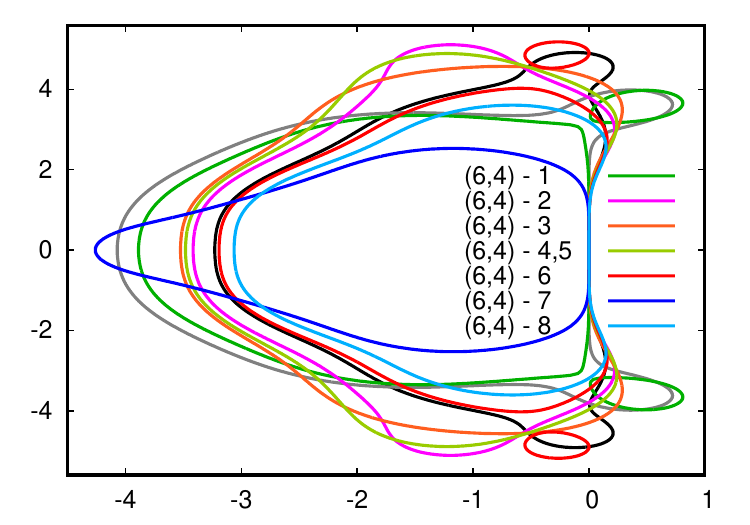}
\caption{Scaling of the new (6,4)$_i$, $i=1,\dots8$ methods, shown in color, and the HALERK64 method of Ref.~\cite{ALLAMPALLI20093837}, shown in black and the RK46-NL method of Ref.~\cite{BERLAND20061459}, shown in gray, for test problem 3 (left). Stability regions for all (6,4) methods considered (right). The stability region for HALERK64~\cite{ALLAMPALLI20093837} is shown in black and for RK46-NL~\cite{BERLAND20061459} in gray.\label{fig_64test3}
}
\end{figure}

%% file: sec_concl.tex
The main result presented here is factorization of the (augmented) Butcher tableau for 2N-storage Runge-Kutta methods of Williamson's type, Eq.~(\ref{eq_A_FD}). This is a feature of all general (\textit{i.e.}, excluding special cases, such as $c_2=c_3$, $c_3=c_4$, etc.) methods independently of their order. Examining the properties of the factors revealed what is suggested to be called $c$-reflection symmetry of 2N-storage methods: for a general method of order of accuracy $p\leqslant4$ that admits a $d$-form representation, there is a $c$-reflected method of the same order of accuracy. The transformation relating the augmented Butcher tableaux of the methods is Eq.~(\ref{eq_tA_GAG}). It is proven that the $c$-reflected method satisfies the order conditions at the same order in Theorem~\ref{th_main}.

It appears that the $c$-reflection symmetry is broken at order five. Perhaps, something is to be learned from the fact that it is broken by the order condition that cannot be represented in the matrix form, similar to Eqs.~(\ref{eq_oc_RK_b_P})--(\ref{eq_oc_RK_ba2c_P}). The numerical evidence about 2N-storage methods of order five is based on the single fifth-order method available in the literature~\cite{Yan2017}. If/when more fifth-order methods are available, their $c$-reflection properties need to be examined.

The scaling of the error for several 2N-storage Runge-Kutta methods of order four found in the literature and their $c$-reflected counterparts is numerically tested for three benchmark problems. The takeaway is that it is beneficial to examine $c$-reflected methods, since they are trivial to construct for existing methods, as they may provide better scaling in particular situations. When developing 2N-storage methods by numerically solving the 2N-storage order conditions, checking simultaneously the $c$-reflected methods doubles the number of available solutions at no cost.

Next, branches of solutions for (5,4) 2N-storage methods, first explored in Ref.~\cite{CK1994}, are constructed numerically with Newton-Raphson method, based on the formulation of the 2N-storage order conditions of Ref.~\cite{Bazavov2025a}. As expected, they respect the $c$-reflection symmetry, as shown in Fig.~\ref{fig_CKcurves}. Four new (5,4) schemes that can be expressed in radicals and one with rational coefficients have been found and examined. While they do not have a desirable property of increasing nodes $c_i$, they still perform reasonably well on the three test problems considered. Experimentation with (6,4) methods revealed several schemes with usable properties. Some of them can be expressed in rationals or radicals. Most importantly, it is established that self-$c$-reflected (6,4) and (8,4) schemes exist, while there seem to be no such (5,4) scheme.

In the absence of closed-form analytic solutions for 2N-storage methods of order four and above, the symmetries and new properties described here and in Ref.~\cite{Bazavov2025a} can help in development and optimization of these schemes.
It may be valuable to explore if any other Runge-Kutta methods possess properties similar to the ones found in the 2N-storage methods of Williamson's type, such as, most importantly, factorization of the (augmented) Butcher tableau.

%% file: sec_app_oc.tex
The order conditions for a Runge-Kutta method to be globally third order of accuracy are:
\begin{eqnarray}
\sum_i b_i &=&1,\label{eq_oc_RK_b}\\
\sum_i b_ic_i &=&\frac{1}{2},\label{eq_oc_RK_bc}\\
\sum_i b_ic_i^2 &=&\frac{1}{3},\label{eq_oc_RK_bc2}\\
\sum_i\sum_j b_i a_{ij} c_j &=&\frac{1}{6}.\label{eq_oc_RK_bac}
\end{eqnarray}

Additional order conditions for a Runge-Kutta method to be globally fourth order are:
\begin{eqnarray}
\sum_i b_ic_i^3 &=&\frac{1}{4},\label{eq_oc_RK_bc3}\\
\sum_i\sum_j b_i c_i a_{ij} c_j &=&\frac{1}{8},\label{eq_oc_RK_bcac}\\
\sum_i\sum_j b_i a_{ij} c_j^2 &=&\frac{1}{12},\label{eq_oc_RK_bac2}\\
\sum_i\sum_j\sum_k b_i a_{ij}a_{jk} c_k &=&\frac{1}{24}.\label{eq_oc_RK_ba2c}
\end{eqnarray}

It is understood that in summations in Eqs.~(\ref{eq_oc_RK_b})--(\ref{eq_oc_RK_ba2c}) the indices run over all possible values where the corresponding coefficient is not zero.

%% file: sec_app_dform.tex
The first order condition, Eq.~(\ref{eq_oc_RK_b}) is fulfilled by construction of the $d$-form. The other seven for a five-stage method of order four are given below in the $d$-form. There is certain regularity, namely, in the monomials the power of the $d_i$ coefficients is either 0 or 1. For order conditions that involve single $b_i$ the $d$-form order conditions contain terms with $d_i$, for those that involve $b_i$ and $a_{ij}$, the lowest order terms are $d_id_j$, and so on. The order of the conditions is carried by the total power of the coefficients $c_i$ in the monomial (keeping in mind that $1\equiv c_6$ for a five-stage method). A minor limitation of the $d$-form of the order conditions is that it only exists for the general case, and it does not capture special cases such as $c_2=c_3$, $c_3=c_4$, etc.

\begin{eqnarray}
\label{eq_oc_d_bc}
\frac{1}{2}&=&c_2(1 - c_2)d_2 + c_3(1 - c_3)d_3 + c_4(1 - c_4)d_4 + c_5(1 - c_5)d_5\nonumber\\
&-& c_2(1 - c_3)d_2d_3 - c_2(1 - c_4)d_2d_4 - c_2(1 - c_5)d_2d_5\nonumber\\
&-& c_3(1 - c_4)d_3d_4 - c_3(1 - c_5)d_3d_5 - c_4(1 - c_5)d_4d_5\nonumber\\
&+& c_2(1 - c_4)d_2d_3d_4 + c_2(1 - c_5)d_2d_3d_5
+ c_2(1 - c_5)d_2d_4d_5 + c_3(1 - c_5)d_3d_4d_5\nonumber\\
&-& c_2(1 - c_5)d_2d_3d_4d_5.
\end{eqnarray}
\begin{eqnarray}
\label{eq_oc_d_bc2}
\frac{1}{3}&=&c_2^2(1 - c_2)d_2 + c_3^2(1 - c_3)d_3 + c_4^2(1 - c_4)d_4 
+ c_5^2(1 - c_5)d_5\nonumber\\
&-& c_2^2(1 - c_3)d_2d_3 - c_2^2(1 - c_4)d_2d_4 - c_2^2(1 - c_5)d_2d_5\nonumber\\
&-& c_3^2(1 - c_4)d_3d_4 - c_3^2(1 - c_5)d_3d_5 - c_4^2(1 - c_5)d_4d_5\nonumber\\
&+& c_2^2(1 - c_4)d_2d_3d_4 + c_2^2(1 - c_5)d_2d_3d_5 + c_2^2(1 - c_5)d_2d_4d_5 + c_3^2(1 - c_5)d_3d_4d_5\nonumber\\
&-& c_2^2(1 - c_5)d_2d_3d_4d_5. 
\end{eqnarray}
\begin{eqnarray}
\label{eq_oc_d_bac}
\frac{1}{6} &=& c_2(c_3 - c_2)(1 - c_3)d_2d_3 + c_2(c_4 - c_2)(1 - c_4)d_2d_4
+ c_2(c_5 - c_2)(1 - c_5)d_2d_5\nonumber\\
&+& c_3(c_4 - c_3)(1 - c_4)d_3d_4 + c_3(c_5 - c_3)(1 - c_5)d_3d_5
+ c_4(c_5 - c_4)(1 - c_5)d_4d_5\nonumber\\
&-& 
c_2(c_4 - c_2)(1 - c_4)d_2d_3d_4 - c_2(c_5 - c_2)(1 - c_5)d_2d_3d_5\nonumber\\
&-& c_2(c_5 - c_2)(1 - c_5)d_2d_4d_5 - c_3(c_5 - c_3)(1 - c_5)d_3d_4d_5\nonumber\\
&+& c_2(c_5 - c_2)(1 - c_5)d_2d_3d_4d_5.
\end{eqnarray}
\begin{eqnarray}
\label{eq_oc_d_bc3}
\frac{1}{4} &=& c_2^3(1 - c_2)d_2 + c_3^3(1 - c_3)d_3 + c_4^3(1 - c_4)d_4 + c_5^3(1 - c_5)d_5\nonumber\\
&-& c_2^3(1 - c_3)d_2d_3 - c_2^3(1 - c_4)d_2d_4 - c_2^3(1 - c_5)d_2d_5\nonumber\\
&-& c_3^3(1 - c_4)d_3d_4 - c_3^3(1 - c_5)d_3d_5 - c_4^3(1 - c_5)d_4d_5\nonumber\\
&+& c_2^3(1 - c_4)d_2d_3d_4 + c_2^3(1 - c_5)d_2d_3d_5 + c_2^3(1 - c_5)d_2d_4d_5 + 
c_3^3(1 - c_5)d_3d_4d_5\nonumber\\
&-& c_2^3(1 - c_5)d_2d_3d_4d_5.
\end{eqnarray}
\begin{eqnarray}
\label{eq_oc_d_bcac}
\frac{1}{8} &=& c_2c_3(c_3 - c_2)(1 - c_3)d_2d_3 + c_2c_4(c_4 - c_2)(1 - c_4)d_2d_4 + 
c_2c_5(c_5 - c_2)(1 - c_5)d_2d_5\nonumber\\
&+& c_3c_4(c_4 - c_3)(1 - c_4)d_3d_4 + 
c_3c_5(c_5 - c_3)(1 - c_5)d_3d_5 + c_4c_5(c_5 - c_4)(1 - c_5)d_4d_5\nonumber\\
&-& [c_2c_3(c_3 - c_2)(1 - c_4) + c_2c_4(c_4 - c_3)(1 - c_4)]d_2d_3d_4\nonumber\\
&-& [c_2c_3(c_3 - c_2)(1 - c_5) + c_2c_5(c_5 - c_3)(1 - c_5)]d_2d_3d_5\nonumber\\
&-& [c_2c_4(c_4 - c_2)(1 - c_5) + c_2c_5(c_5 - c_4)(1 - c_5)]d_2d_4d_5\nonumber\\
&-& [c_3c_4(c_4 - c_3)(1 - c_5) + c_3c_5(c_5 - c_4)(1 - c_5)]d_3d_4d_5\nonumber\\
&+& [c_2c_3(c_3 - c_2)(1 - c_5) + c_2c_4(c_4 - c_3)(1 - c_5) + 
c_2c_5(c_5 - c_4)(1 - c_5)]d_2d_3d_4d_5.
\end{eqnarray}
\begin{eqnarray}
\label{eq_oc_d_bac2}
\frac{1}{12} &=&c_2^2(c_3 - c_2)(1 - c_3)d_2d_3 + c_2^2(c_4 - c_2)(1 - c_4)d_2d_4 + 
c_2^2(c_5 - c_2)(1 - c_5)d_2d_5\nonumber\\
&+& c_3^2(c_4 - c_3)(1 - c_4)d_3d_4 + c_3^2(c_5 - c_3)(1 - c_5)d_3d_5 
+ c_4^2(c_5 - c_4)(1 - c_5)d_4d_5\nonumber\\
&-& c_2^2(c_4 - c_2)(1 - c_4)d_2d_3d_4 - c_2^2(c_5 - c_2)(1 - c_5)d_2d_3d_5\nonumber\\
&-& c_2^2(c_5 - c_2)(1 - c_5)d_2d_4d_5 - c_3^2(c_5 - c_3)(1 - c_5)d_3d_4d_5\nonumber\\
&+& c_2^2(c_5 - c_2)(1 - c_5)d_2d_3d_4d_5.
\end{eqnarray}
\begin{eqnarray}
\label{eq_oc_d_ba2c}
\frac{1}{24} &=& c_2(c_3 - c_2)(c_4 - c_3)(1 - c_4)d_2d_3d_4 
+ c_2(c_3 - c_2)(c_5 - c_3)(1 - c_5)d_2d_3d_5\nonumber\\
&+& c_2(c_5 - c_4)(c_4 - c_2)(1 - c_5)d_2d_4d_5 
+ c_3(c_4 - c_3)(c_5 - c_4)(1 - c_5)d_3d_4d_5\nonumber\\
&-& [c_2(c_3 - c_2)(c_5 - c_3)(1 - c_5) + 
c_2(c_5 - c_4)(c_4 - c_3)(1 - c_5)]d_2d_3d_4d_5.
\end{eqnarray}

%% file: sec_app_checks54.tex
It was found in the course of this work that it is convenient to search for numerical solutions of the order conditions in the basis of the Butcher tableau coefficients $a_{ij}$, $b_i$, $c_i$, as advocated in Ref.~\cite{Bazavov2025a}. For (5,4) methods there are 15 variables and 14 2N-storage order conditions. For Groebner basis manipulations, however, this system of equations is too computationally expensive, and the $d$-form of the order conditions proved more beneficial. The goal is to minimize the number of variables and total monomial powers. For a (5,4) method there are nine non-trivial parameters once 2N-storage constraints are explicitly taken into account, either $A_2,\dots A_5$, $B_1,\dots,B_5$ or $c_2,\dots,c_5$, $b_1,\dots,b_5$ if Eq.~(\ref{eq_a_bc_rec}) is used. In the $d$-form $d_1=B_1/c_2=1$, so there is one variable less, compensated by the fact that the first-order conditions, Eq.~(\ref{eq_oc_RK_b}) is fulfilled by construction. Thus, there are seven equations (\ref{eq_oc_d_bc})--(\ref{eq_oc_d_ba2c}), listed in \ref{sec_app_dform}, and eight variables $c_2,\dots,c_5$, $d_2,\dots,d_5$. By inspection one finds that in the (5,4) $d$-form order conditions $d_2$ is always accompanied by $c_2$ and $d_5$ by $1-c_5$. Thus, introducing $d'_2=c_2d_2$ and $d'_5=(1-c_5)d_5$ ($\equiv B_5$) decreases some of the monomial powers by one or two orders. In the new form, the order conditions are manageable for, \textit{e.g.}, the Singular package~\cite{DGPS}. There is still one free parameter. More simplifications are achieved when the free parameter is taken as $c_2=1$ or $\{c_3,c_4,c_5\}=\{0,1\}$. It takes 2-3 hours on a laptop for Singular to reduce the system to a triangular form. Those equations are then quickly solved with Mathematica~\cite{Mathematica}. Next, one can fix the free parameter at some arbitrary value, in particular, $c_4=\{2/3,3/5,4/7,5/8,5/9\}$ were solved. Those cases take up to 6 hours to reduce the system. About 120 solutions were found in this way. All of them landed on the already computed branches of solutions shown in Fig.~\ref{fig_CKcurves}. Apart from validating the solutions computed directly with the Newton-Raphson method on the original order conditions, this shows that the $d$-form order conditions may be advantageous for semi-analytic work.

%% file: sec_app_matlab.tex
The Matlab script listed below computes the $c$-reflected pairs for several fourth-order 2N-storage Runge-Kutta methods available in the literature. The order conditions are checked deliberately in the standard form, Eqs.~(\ref{eq_oc_RK_b})--(\ref{eq_oc_RK_ba2c}).

\begin{verbatim}
% construct c-reflected methods and check the order conditions
function main
  for irpt=1:4
    if irpt == 1
      fprintf( "\nLSRK (5,4) solution 1 of Carpenter-Kennedy\n\n" );
      [ s, A, B ] = setup_CK1();
    elseif irpt == 2
      fprintf( "\nLSRK (6,4) of Berland, Bogey, Bailly\n\n" );
      [ s, A, B ] = setup_BBB();
    elseif irpt == 3
      fprintf( "\nLSRK (8,4) of Toulorge, Desmet\n\n" );
      [ s, A, B ] = setup_TDRKF84();
    elseif irpt == 4
      fprintf( "\nLSRK (14,4) of Niegemann, Diehl, Busch\n\n" );
      [ s, A, B ] = setup_NDB14();
    end
    [ a, b, c ] = convert_A_to_a( s, A, B );
    [ alpha, beta ] = convert_a_to_alpha( s, a, b );
    fprintf( "Original A, B:\n" );
    for i=1:s
      fprintf( "A%d = %f  B%d = %f\n", i, A(i), i, B(i) );
    end
    fprintf( "\n" );
    fprintf( "Original Butcher tableau\n" );
    print_butcher( s, a, b, c );
    fprintf( "\n" );
    check_oc_standard_4( s, a, b, c )
    check_oc_lowstorage( s, alpha, beta )
    fprintf( "\n" );
    augA = convert_a_to_augA( s, a, b );
    [ G, T ] = form_G_T( s, B, c );
    D = inv(G);
    % c-reflected method
    augAt = T*transpose(D*augA*G)*T;
    [ at, bt, ct ] = convert_augA_to_a( s, augAt );
    [ alphat, betat ] = convert_a_to_alpha( s, at, bt );
    [ At, Bt ] = convert_a_to_A( s, at, bt, ct );
    fprintf( "c-reflected A, B:\n" );
    for i=1:s
      fprintf( "A%d = %f  B%d = %f\n", i, At(i), i, Bt(i) );
    end
    fprintf( "\n" );
    fprintf( "c-reflected Butcher tableau\n" );
    print_butcher( s, at, bt, ct );
    fprintf( "\n" );
    check_oc_standard_4( s, at, bt, ct )
    check_oc_lowstorage( s, alphat, betat )
    fprintf( "\n" );
  end
end
% (5,4) solution 1 of Carpenter, Kennedy
function [ s, A, B ] = setup_CK1()
  s = 5;
  A = [ 0 ...
       -0.4812317431372 ...
       -1.049562606709 ...
       -1.602529574275 ...
       -1.778267193916 ];
  B = [ 9.7618354692056E-2 ...
        0.4122532929155 ...
        0.4402169639311 ...
        1.426311463224 ...
        0.1978760537318 ];
end
% (6,4) of Berland, Bogey, Bailly
function [ s, A, B ] = setup_BBB()
  s = 6;
  A = [ 0 ...
       -0.737101392796 ...
       -1.634740794341 ...
       -0.744739003780 ...
       -1.469897351522 ...
       -2.813971388035 ];
  B = [ 0.032918605146 ...
        0.823256998200 ...
        0.381530948900 ...
        0.200092213184 ...
        1.718581042715 ...
        0.27 ];
end
% (8,4) of Toulorge, Desmet
function [ s, A, B ] = setup_TDRKF84()
  s = 8;
  A = [ 0 ...
       -0.5534431294501569 ...
        0.01065987570203490 ...
       -0.5515812888932000 ...
       -1.885790377558741 ...
       -5.701295742793264 ...
        2.113903965664793 ...
       -0.5339578826675280 ];
  B = [ 0.08037936882736950 ...
        0.5388497458569843 ...
        0.01974974409031960 ...
        0.09911841297339970 ...
        0.7466920411064123 ...
        1.679584245618894 ...
        0.2433728067008188 ...
        0.1422730459001373 ];
end
% (14,4) of Niegemann, Diehl, Busch
function [ s, A, B ] = setup_NDB14()
  s = 14;
  A = [ 0 ...
       -0.7188012108672410 ...
       -0.7785331173421570 ...
       -0.0053282796654044 ...
       -0.8552979934029281 ...
       -3.9564138245774565 ...
       -1.5780575380587385 ...
       -2.0837094552574054 ...
       -0.7483334182761610 ...
       -0.7032861106563359 ...
        0.0013917096117681 ...
       -0.0932075369637460 ...
       -0.9514200470875948 ...
       -7.1151571693922548 ];
  B = [ 0.0367762454319673 ...
        0.3136296607553959 ...
        0.1531848691869027 ...
        0.0030097086818182 ...
        0.3326293790646110 ...
        0.2440251405350864 ...
        0.3718879239592277 ...
        0.6204126221582444 ...
        0.1524043173028741 ...
        0.0760894927419266 ...
        0.0077604214040978 ...
        0.0024647284755382 ...
        0.0780348340049386 ...
        5.5059777270269628 ];
end
function [ a, b, c ] = convert_A_to_a( s, A, B )
  a = zeros(s,s);
  b = zeros(s,1);
  c = zeros(s,1);
  for i=2:s
    a(i,i-1) = B(i-1);
    c(i) = a(i,i-1);
    for j=i-2:-1:1
      a(i,j) = A(j+1)*a(i,j+1) + B(j);
      c(i) = c(i) + a(i,j);
    end
  end
  b(s) = B(s);
  for i=s-1:-1:1
    b(i) = A(i+1)*b(i+1) + B(i);
  end
end
function [ alpha, beta ] = convert_a_to_alpha( s, a, b )
  alpha = zeros(s,s);
  beta = zeros(s,1);
  for i=2:s
     for j=1:i-1
       alpha(i,j) = a(i,j) - a(i-1,j);
     end
  end
  for i=1:s
    beta(i) = b(i) - a(s,i);
  end
end
function [ A, B ] = convert_a_to_A( s, a, b, c )
  A = zeros(s,1);
  B = zeros(s,1);
  A(1) = 0;
  for i=2:s
    A(i) = ( b(i-1) - a(s,i-1) ) / ( b(i) - a(s,i) );
  end
  for i=1:s-1
      B(i) = a(i+1,i);
  end
  B(s) = b(s);
end
function augA = convert_a_to_augA( s, a, b )
  augA = zeros( s+1, s+1 );
  for i=2:s
    for j=1:i-1
      augA(i,j) = a(i,j);
    end
    augA(s+1,i) = b(i);
  end
  augA(s+1,1) = b(1);
end
function [ a, b, c ] = convert_augA_to_a( s, augA )
  a = zeros(s,s);
  b = zeros(s,1);
  c = zeros(s,1);
  for i=2:s
    for j=1:i-1
      a(i,j) = augA(i,j);
      c(i) = c(i) + a(i,j);
    end
    b(i) = augA(s+1,i);
  end
  b(1) = augA(s+1,1);
end
function [ G, T ] = form_G_T( s, B, c )
  G = tril( ones( s+1, s+1 ) );
  for i=2:s
    if i==s
      G(i,i) = ( 1 - c(i) ) / B(i);
    else
      G(i,i) = ( c(i+1) - c(i) ) / B(i);
    end
  end
  T = zeros( s+1, s+1 );
  for i=1:s+1
    T( s+2-i, i ) = 1;
  end
end
function check_oc_standard_4( s, a, b, c )
  fprintf( "Standard order conditions:\n" );
  sum = 0;
  for i=1:s
    sum = sum + b(i);
  end
  fprintf( "b_i               - 1    = %e\n", sum-1 );
  sum = 0;
  for i=2:s
    sum = sum + b(i)*c(i);
  end
  fprintf( "b_i*c_i           - 1/2  = %e\n", sum-1/2 );
  sum = 0;
  for i=2:s
    sum = sum + b(i)*c(i)^2;
  end
  fprintf( "b_i*c_i^2         - 1/3  = %e\n", sum-1/3 );
  sum = 0;
  for i=3:s
    for j=2:i-1
      sum = sum + b(i)*a(i,j)*c(j);
    end
  end
  fprintf( "b_i*a_ij*c_j      - 1/6  = %e\n", sum-1/6 );
  sum = 0;
  for i=2:s
    sum = sum + b(i)*c(i)^3;
  end
  fprintf( "b_i*c_i^3         - 1/4  = %e\n", sum-1/4 );
  sum = 0;
  for i=3:s
    for j=2:i-1
      sum = sum + b(i)*c(i)*a(i,j)*c(j);
    end
  end
  fprintf( "b_i*c_i*a_ij*c_j  - 1/8  = %e\n", sum-1/8 );
  sum = 0;
  for i=3:s
    for j=2:i-1
      sum = sum + b(i)*a(i,j)*c(j)^2;
    end
  end
  fprintf( "b_i*a_ij*c_j^2    - 1/12 = %e\n", sum-1/12 );
  sum = 0;
  for i=4:s
    for j=3:i-1
      for k=2:j-1
        sum = sum + b(i)*a(i,j)*a(j,k)*c(k);
      end
    end
  end
  fprintf( "b_i*a_ij*a_jk*c_k - 1/24 = %e\n", sum-1/24 );
end
function check_oc_lowstorage( s, alpha, beta )
  fprintf( "2N-storage order conditions:\n" );
  for j=1:s-2
    for i=j+2:s
      sum = beta(j+1)*alpha(i,j) - beta(j)*alpha(i,j+1);
      fprintf( "beta%d*alpha%d%d - beta%d*alpha%d%d = %e\n", ...
        j+1, i, j, j, i, j+1, abs(sum) );
    end
  end
end
function print_butcher( s, a, b, c )
  for i=2:s
    fprintf( "%.6f  |", c(i) );
    for j=1:i-1
        fprintf( "  %.6f", a(i,j) );
    end
    fprintf( "\n" );
  end
  fprintf( "----------|" );
  for i=1:s
    fprintf( "----------" );
  end
  fprintf( "\n" );
  fprintf( "          |" );
  for i=1:s
      fprintf( "  %.6f", b(i) );
  end
  fprintf( "\n" );
end
\end{verbatim}